\documentclass[12pt, a4paper]{amsart}
\usepackage{amsmath, amsthm, amstext, amsbsy, amssymb} 
\usepackage{mathrsfs}

\usepackage[backend=biber, style=alphabetic, sorting=nyt, nohashothers=true, maxbibnames=99, maxalphanames=99, giveninits=true]{biblatex}
\usepackage{xcolor}                 
\usepackage[shortlabels]{enumitem} 

\usepackage[a4paper, margin=2.5cm]{geometry}

\usepackage{hyperref} 				
\hypersetup{						
	colorlinks=true,
	linkcolor=blue,
	citecolor=red,
}

\allowdisplaybreaks[1]		   			

\newtheorem{prop}{Proposition}
\newtheorem{lem}{Lemma}
\newtheorem{thrm}{Theorem}

\newtheorem*{thrm*}{Theorem}
\newtheorem{thrmm}{Theorem}
 
\newtheorem{crl}{Corollary}
\newtheorem{rmk}{Remark}

\newtheorem*{clm}{Claim}
\newtheorem*{prm}{Problem}

\theoremstyle{definition}  
\newtheorem{definition}{Definition}

\renewcommand{\qedsymbol}{$ \blacksquare $}

\renewcommand{\le}{\leqslant}
\renewcommand{\ge}{\geqslant}

\let\intt\int
\renewcommand{\int}{\intt\limits}

\newcommand{\C}{\mathbb{C}} 						
\newcommand{\D}{\mathbb{D}} 						
\newcommand{\R}{\mathbb{R}}		                    
\newcommand{\N}{\mathbb{N}}	                	    
\newcommand{\TT}{\mathbb{T}}	 					
\newcommand{\Z}{\mathbb{Z}}	 					    
                        
\newcommand{\F}{\mathcal{F}}                        
\newcommand{\M}{\mathcal{M}}                        
\newcommand{\GG}{\mathcal{G}}

\renewcommand{\a}{\alpha}									
\renewcommand{\b}{\beta}									
\newcommand{\g}{\gamma}						                
\newcommand{\G}{\Gamma}		                                
\newcommand{\de}{\delta}									
\newcommand{\e}{\varepsilon}								
\newcommand{\z}{\zeta}										
\renewcommand{\t}{\theta}									
\renewcommand{\k}{\kappa}									
\renewcommand{\l}{\lambda}									
\renewcommand{\L}{\Lambda}									
\let\originalnu\nu											
\renewcommand{\nu}{\originalnu} 					        
\newcommand{\f}{\varphi}									

\DeclareMathOperator{\dist}{dist}							

\DeclareMathOperator{\supp}{supp}					

\renewcommand{\Re}{\operatorname{Re}}		    	

\DeclareMathOperator{\id}{id}								

\DeclareMathOperator{\Hol}{Hol}

\title[Shift-invariant sampling in two-sided small Fock spaces]
{Shift-invariant sampling in two-sided small Fock spaces}
\author{Yurii Belov \and Mikhail Mironov}
\thanks{The work was carried out with the financial support of MSHE RF in the framework of a scientific project under agreement no. 075-15-2025-013.}
\address{
\newline \phantom{x}\,\, Yurii Belov,
\newline Department of Mathematics and Computer Science
St. Petersburg State University,
14th Line 29b, Vasilyevsky Island, St. Petersburg, Russia, 199178,
\newline {\tt j\_b\_juri\_belov@mail.ru} 
\smallskip
\newline \phantom{x}\,\, Mikhail Mironov,
\newline Univ Gustave Eiffel, Univ Paris Est Créteil, CNRS, LAMA UMR8050 F-77447 Marne-la-Vallée, France
\newline {\tt mikhail.mironov2@univ-eiffel.fr}
\smallskip
}

\begin{document}

\begin{abstract}
We consider the sampling problem for two-sided small Fock spaces $\mathcal{F}^p_{\alpha}$, for the full range $0 < p \le \infty$. 
We establish a geometric description of shift-invariant sampling sequences, i.e., sequences $\Lambda$ such that $c \Lambda$ is sampling for all $c \in \mathbb{C} \setminus \{ 0 \}$.
\end{abstract} 

\maketitle

\section{Introduction}
    \subsection{Definitions and main results}
        In this section, we introduce the basic concepts and formulate the main results. We define two-sided small Fock spaces, establish what it means for a sequence to be shift-invariant sampling, and present an appropriate notion of density in this setting. Next, we state the central result of this paper, a description of shift-invariant sampling sequences in terms of lower density.
    
    \subsubsection{Small Fock spaces}
        We consider an analogue of Fock-type spaces with a singularity at the origin.
        For brevity, we denote $\C_0 = \C \setminus \{ 0 \}$.
        \begin{definition}
            Let $\f: \C_0 \to \R$ be a radial function $\f(z) = \f(|z|)$ (a weight) that is decreasing on $(0, 1]$ and increasing on $[1, \infty)$. 
        
            \emph{The two-sided Fock-type space} for this weight $\f$ and $0 < p \le \infty$ is
            \begin{align*}
                \F^{\infty}_{\f} & = \{ f \in \Hol(\C_0): \|f\|_{\F^{\infty}_{\f}} = \sup_{z \in \C_0} |f(z)|e^{-\f(z)} < \infty \}, \quad p = \infty, \\
                \F^{p}_{\f} & = \{ f \in \Hol(\C_0): \|f\|^p_{\F^{p}_{\f}} = \int_{\C_0} |f(z)|^pe^{-p\f(z)} dA(z) < \infty \}, \quad 0 < p < \infty,
            \end{align*}     
            where $dA$ is the two-dimensional Lebesgue measure.
        \end{definition}
        
        For $\f(z) = \a \log^2|z|, \ \a > 0$, we call such spaces \emph{two-sided small Fock spaces}, and denote them $\F^{p}_\a$. These are Banach spaces for $1 \le p \le \infty$ and F-spaces for $0 < p < 1$. 
        \emph{One-sided small Fock spaces} are their subspaces which consist of entire functions. The one-sided case is extensively explored in the literature; see, e.g., Baranov, Dumont, Hartmann and Kellay \cite{SmallFock}, Omari \cite{OmariCIS}, Omari and Kellay \cite{OmariCISb} for recent developments and an overview of the previous results.
        
        We consider spaces of functions with a singularity at the origin since that appears to be natural with the $\a \log^2|z|$ weight. In particular, that provides us with automorphisms of the form $f(z) \mapsto a_n z^{n} f(b^n z), \ n \in \Z$, see Subsection \ref{subsect: Shift-invariance}, and a connection to Gabor analysis, see Subsection \ref{sect: Gabor}.

    \subsubsection{Sampling and shift-invariant sampling}
        \begin{definition}
            A set $\L \subset \C_0$ is \emph{a sampling set for} $\F^{\infty}_\a$ if there exists $K > 0$ such that for any $f \in \F^{\infty}_\a$ 
            \begin{equation*}
                \| f \|_{\F^{\infty}_{\a}} \le K \sup_{\l \in \L} |f(\l)|e^{-\a \log^2|\l|}.
            \end{equation*}
            In this case we say that $\L$ is \emph{SS} for $\F^{\infty}_\a$.
            We denote $\| f|_{\L} \|_{\infty, \a} = \sup_{\l \in \L} |f(\l)|e^{-\a \log^2|\l|}$.
            The sampling constant for a set $\L$ is
            \begin{equation} \label{eq: K Lambda definiton}
                K_{\L} = \sup_{f \in \F^{\infty}_\a \setminus \{ 0 \}} \frac{\| f \|_{\F^{\infty}_\a}}{\| f|_{\L} \|_{\infty, \a}}.
            \end{equation}
        \end{definition}
        Clearly, $\L$ being a sampling set for $\F^{\infty}_\a$ is equivalent to $K_\L < \infty$.
        \begin{definition}
            Let $0 < p < \infty$. A set $\L \subset \C_0$ is \emph{a sampling set for} $\F^{p}_\a$ if there exist $A, B > 0$ such that for any $f \in \F^{p}_\a$     
            \begin{equation} \label{eq: sampling definition finite p}
                A \| f|_{\L} \|_{p, \a}^p \le \| f \|_{\F^p_{\a}}^p \le B \| f|_{\L} \|_{p, \a}^p,
            \end{equation}
            where we define
            \begin{equation} \label{eq: restriction norm}
                \| f|_{\L} \|_{p, \a}^p = \sum_{\l \in \L} |f(\l)|^p |\l|^2 e^{-p\a \log^2 |\l|}.   
            \end{equation}
            In this case we say that $\L$ is \emph{SS} for $\F^{p}_\a$.
            The sampling constants for a set $\L$ are
            \begin{equation} \label{eq: A B Lambda definiton}
                A_{\L} = \inf_{f \in \F^{p}_\a \setminus \{ 0 \}} \frac{\| f \|_{\F^p_{\a}}^p}{\| f|_{\L} \|_{p, \a}^p}, \qquad B_{\L} = \sup_{f \in \F^{p}_\a \setminus \{ 0 \}} \frac{\| f \|_{\F^p_{\a}}^p}{\| f|_{\L} \|_{p, \a}^p}.
            \end{equation}
        \end{definition}
        \begin{definition}
            A set $\L \subset \C_0$ is \emph{shift-invariant sampling for} $\F^{\infty}_\a$ if 
            \begin{equation*}
                \sup_{c > 0} K_{c \L} < \infty. 
            \end{equation*}
            In this case we say that $\L$ is \emph{ShS} for $\F^{\infty}_\a$. 
            
            Similarly, let $0 < p < \infty$. A set $\L \subset \C_0$ is \emph{shift-invariant sampling for} $\F^{p}_\a$ if 
            \begin{equation*}
                0 < \inf_{c > 0} A_{c \L} \le \sup_{c > 0} B_{c \L} <\infty. 
            \end{equation*}
            In this case, we say that $\L$ is \emph{ShS} for $\F^{p}_\a$. In fact, $\L$ is ShS for $\F^{p}_\a$ if and only if $c \L$ is SS for $\F^{p}_\a$ for any $c > 0$, see Subsection \ref{subsect: Shift-invariance}, Remark \ref{rmk: ShS is SS for any c}.
            Moreover, since two-sided small Fock spaces are radial, we can always replace $c > 0$ by $c \in \C_0$. 
        \end{definition}

    \subsubsection{Overview}
        For Fock-type spaces of entire functions with $\f(z) \gg \log^2|z|$ and $1 \le p \le \infty$, there exists a complete solution of the sampling problem. For Bargmann-Fock space, $\f(z) = c |z|^2$, it was obtained by Seip \cite{SeipBargFock}, and Seip and Wallstén \cite{SeipBargFock2}. The solution was gradually extended for more general weights, see \cite{lyubarskii1994sampling, Ortega1995OnInterpolation, marco2003interpolating, borichev2007sampling}. Finally, Borichev and Lyubarskii \cite{FockType} demonstrated that this solution does not extend for weights $\f$ that grow slower than $\log^2|z|$. 
        
        For one-sided small Fock spaces with $p = 2, \infty$, there are partial density results, see \cite[Theorems 1.4, 1.5]{SmallFock}. For two-sided small Fock spaces with $p = 2$, there is a density description for ShS sequences $\L$ that lie on a logarithmic spiral, see \cite[Theorem 1.1, Lemma 5.1]{HypCh}.
        
        We note that the definition of shift-invariant sampling is reasonable in the sense that sampling does not automatically imply shift-invariant sampling. It is a consequence of the fact that two-sided small Fock spaces are not shift-invariant themselves. 
        This situation differs from the Bernstein space case, Beurling \cite[Balayage of Fourier-Stieltjes Transforms, p.341]{Beurling}, and Bargmann-Fock space, \cite{SeipBargFock, SeipBargFock2}, where these two notions coincide.
        
        In fact, this is the foundation on which Beurling's balayage technique \cite{Beurling} stands, we refer to \cite[Lecture 3]{olevskii2016functions} for a self-contained presentation of Beurling's technique for the Bernstein space. The shift-invariance is the key that allows one to apply the technique to Bargmann-Fock space as well \cite{SeipBargFock, SeipBargFock2}. In our case, Beurling's ideas allow us to provide a description of shift-invariant sampling sequences in the $p = \infty$ case, and an idea from \cite{SeipBargFock} allows us to extend the description to $0 < p \le \infty$. 

    \subsubsection{Logarithmic density}
        We establish an appropriate notion of separation in our setting.
        Observe that $\R \times \TT \cong \C_0$ via $(t, \z) \mapsto e^t \z$. Hence, one can consider a metric $d_{\log}$ on $\C_0$, the product metric on $\R \times \TT$, that is, for $z, w \in \C_0$
        \begin{equation*}
            d_{\log} (z, w) = |\log|z| - \log|w| | + \left| \frac{z}{|z|} - \frac{w}{|w|} \right|. 
        \end{equation*}
        \begin{definition}
            We call $\L$ \emph{separated} if there exists $d_{\L} > 0$ such that for any two different $\l, \l' \in \L$
            \begin{equation*}
                d_{\log}(\l, \l') \ge d_{\L}.
            \end{equation*}
        \end{definition}
        Similarly to \cite[Theorem 2, p.344]{Beurling} and \cite[Lemma 4.1, Lemma 7.2]{SeipBargFock} it is possible to show that we can go from SS (ShS) $\L$ to its separated subset without losing the SS (ShS) property. In light of this, we formulate our main results for separated sequences. 
        \begin{definition}
            For a set $\L \subset \C_0$, we define \emph{the lower logarithmic density} 
            \begin{equation*}
                D^-_{\log}(\L) = \liminf_{R \to \infty} \inf_{c > 0} \frac{|\{ \l \in \L: c \le |\l| \le c e^R\}|}{R}.
            \end{equation*}
            If we denote $\L$ as a subset of $\R \times \TT$ by $L$, then this density takes a more classic form 
            \begin{equation*}
                D^-_{\log}(\L) = \liminf_{R \to \infty} \inf_{t \in \R} \frac{|L \cap [t, t + R] \times \TT|}{R}.
            \end{equation*}
        \end{definition}

    \subsubsection{Main results}
        Finally, with the right notion of density, we are ready to formulate the main result of this paper. We separate it into two theorems for convenience.
        \begin{thrm} \label{thrm: ShS p = infty}
            Suppose $\L \subset \C_0$ is separated. Then $\L$ is shift-invariant sampling for $\F^{\infty}_\a$ if and only if $D^-_{\log}(\L) > 2\a$.
        \end{thrm}
        \begin{thrm} \label{thrm: ShS p = 2}
            Let $0 < p < \infty$. Suppose $\L \subset \C_0$ is separated. Then $\L$ is shift-invariant sampling for $\F^{p}_\a$ if and only if $D^-_{\log}(\L) > 2\a$.
        \end{thrm}
        One of the novelties of this result is that Beurling's technique together with Seip's idea allow us to extend the description to $0 < p \le \infty$. 
        These two theorems can be extended to cover non-separated sequences similarly to \cite[Theorem 2.1, Theorem 2.3]{SeipBargFock}.  
        \begin{rmk}
            Let $0 < p < \infty$. Then $\L$ is shift-invariant sampling for $\F^{p}_\a$ if and only if $\L$ can be represented as a finite union of separated sequences and there exists a separated sequence $\L' \subset \L$ such that $D^-_{\log}(\L') > 2\a$. A set $\L$ is shift-invariant sampling for $\F^{\infty}_\a$ if and only if there exists a separated sequence $\L' \subset \L$ such that $D^-_{\log}(\L') > 2\a$. 
        \end{rmk}
        The authors do not know whether there exists a geometric description for sampling sequences.
        \begin{prm}
            Provide a geometric description of sampling sequences in one-sided and two-sided small Fock spaces.
        \end{prm}
        However, it should be noted that at least a partial solution is attainable for one-sided small Fock spaces.
        
        In the classic Paley-Wiener scenario, the problem of describing sampling sequences had been standing for a long time. It was stated in 1952 by Duffin and Schaeffer \cite{DS52} as a question about Fourier frames. Finally, in 2002, Seip and Ortega-Cerdà \cite[Theorem 1]{FourierFrames} obtained an answer in terms of the de Branges theory. Essentially, it states that $\L$ is sampling if and only if the Paley-Wiener space, as a de Branges space, can be embedded into a larger de Branges space for which $\L$ is complete interpolating. This condition is less explicit than a density condition.
        
        We recall that one-sided small Fock spaces with $p = 2$ are de Branges spaces \cite[Proposition 2.8]{SmallFock}. Therefore, it is possible to obtain a description of real sampling sequences in the spirit of \cite{FourierFrames}. The authors do not know whether this result can be used to obtain a geometric description, at least for real sampling sequences, in the $p = 2$ case.

    \subsubsection*{The idea of the proof} 
        \begin{definition}
            Let $0 < p \le \infty$. We say that $\L$ is \emph{a complete interpolating sequence for} $\F^p_\a$ if $\L$ is SS for $\F^p_\a$ and the restriction operator $f \mapsto (f(\l)|\l|^{2/p}e^{-\f(\l)})_{\l \in \L}$ is an isomorphism between $\F^p_\a$ and $\ell^p(\L)$.
        \end{definition}
        The main idea in \cite{SeipBargFock} is that there are no such sequences for Bargmann-Fock spaces. In contrast, there are plenty of complete interpolating sequences for two-sided small Fock spaces. In particular, such sequences distinguish the notions of sampling and shift-invariant sampling.
        We use a geometric description of complete interpolating sequences, \cite[Theorem 2]{SmallFockCIS}, which is similar to the one-sided case due to Baranov, Dumont, Hartmann and Kellay \cite[Theorem 1.1, Theorem 1.2]{SmallFock} and Omari \cite[Theorem 1, Theorem 2]{OmariCIS}.
        \begin{thrmm} \label{thrm: CIS characterisation}
            Let $0 < p \le \infty$. A set $\L \subset \C_0$ is complete interpolating for $\F^p_{\a}$ if and only if the following three conditions are satisfied:
            \begin{enumerate} [label=(\roman*)]
                \item $\L$ is $d_{\log}$-separated. \label{pr: 1 two-sided}
            \end{enumerate} 
            In particular, we can enumerate $\L = ( \l_k )_{k \in \Z}$ so that $|\l_k| \le |\l_{k + 1}|$ and write $$\l_k = e^{\frac{k + 2/p + \de_k}{2\a}} e^{i \theta_k}, \ \de_k, \theta_k \in \R.$$
            \begin{enumerate} [label=(\roman*)]
                \addtocounter{enumi}{1} 
                \item $(\de_k)_{k \in \Z} \in \ell^{\infty}(\Z)$ \label{pr: 2 two-sided}
                \item Up to a shift in numeration there exists $N \in \N$ and $\e > 0$ such that 
                \begin{equation*}
                    \sup_{n \in \Z} \left| \frac{1}{N} \sum_{k = n}^{n + N - 1} \de_k \right| \le \frac{1}{2} - \e.
                \end{equation*} \label{pr: 3 two-sided}
            \end{enumerate}
        \end{thrmm}
        It is a perturbation-type description that resembles the famous 1/4 Kadets–Ingham theorem for the Paley-Wiener space. 
        \begin{rmk} \label{rmk: m shift}
            Note that instead of saying that there exists a shift in numeration we can explicitly say that there exists $m \in \Z$, $N \in \N$ and $\e > 0$ such that
            \begin{equation*}
                \sup_{n \in \Z} \left| \frac{1}{N} \sum_{k = n + 1}^{n + N} \de_k + m \right| \le \frac{1}{2} - \e.    
            \end{equation*}
            Moreover, such $m$, if exists, is unique and, hence, the right shift in numeration is also unique.
        \end{rmk}
        The idea from \cite{SeipBargFock}, in fact, works in the sense that there are no shift-invariant complete interpolating sequences.
        
        In light of the density description of ShS, it is possible to show, similarly to the one-sided case \cite[Theorem 1.3]{SmallFock}, that any ShS sequence contains a complete interpolating subsequence. 
        
        We also note that any $\L$, complete interpolating sequence for $\F^p_{\a}$, has the critical density $D^-_{\log}(\L) = 2\a$. Hence, any such sequence is SS but not ShS. 

    \subsection{Motivation through Gabor frames} \label{sect: Gabor}
        Before proceeding to the main part, we present the original motivation for this paper. The motivation comes from Gabor analysis, and specifically the Gabor frame problem, a fundamental question concerning the completeness and stability of time–frequency representations \cite[Proposition 5.2.1]{FoundTimeFreq}.
        The reader may skip this section without loss of continuity.   
        
        We move into the setting of Gabor analysis, which is a type of time-frequency analysis. 
        The \emph{time-frequency} shift of a function $g \in L^2(\R)$ is
        \begin{equation*}
            M_yT_xg(t) = e^{2\pi i y t} g(t - x).
        \end{equation*}
        The decomposition of a signal ($L^2$ function) with respect to time-frequency shifts of a given \emph{window function} $g$ is the basis of time-frequency analysis.
        
        Suppose $S$ is a discrete subset of the time-frequency plane $\R^2$. The corresponding \emph{Gabor system} is
        \begin{equation*}
            \GG(g, S) = \{ M_yT_xg: (x,y) \in S \}.
        \end{equation*}
        \begin{definition}
            We say that $\GG(g, S)$ is \emph{a frame in} $L^2(\R)$ whenever there are $A, B > 0$ such that
            \begin{equation*}
                A \|f\|^2 \le \sum_{(x, y) \in S} |\langle f, M_yT_xg \rangle|^2 \le B \|f\|^2, \quad f \in L^2(\R).
            \end{equation*}
        \end{definition}
        A classical example is the rectangular lattice $S = \a \Z \times \b \Z, \ \a, \b > 0$. So, the problem is to find such $(\a, \b)$ that $\GG(g, \a \Z \times \b \Z)$ is a frame, where the window $g$ is known. We refer to \cite{FoundTimeFreq} for an in-depth exposition on the topic. It is well known that $\a \b \le 1$ is a necessary condition for $\GG(g, \a \Z \times \b \Z)$ to be a frame \cite[Corollary 7.5.1]{FoundTimeFreq}.
        
        Even in this lattice case, the problem is completely solved only for very specific functions or function classes, see e.g. \cite{ChFrames,IntervalFrames,HaarFrames,ComplPositive, ShSGabor, RatFrames} and references therein. 

        For a non-lattice set of time-frequency shifts $S$, that is, for the irregular case, almost nothing is known. The problem of Gabor frames in such generality is solved only for the Gaussian.
        
        The semi-regular case $S = \L \times \b \Z$ is more accessible. Gröchenig, Romero and Stöckler obtained a solution for totally positive functions of Gaussian type \cite{ShSGabor}, and for the hyperbolic secant \cite{HypChReal}. Recently, in 2024, Belov and Baranov \cite{HypCh} solved the problem for hyperbolic secant-type windows, that is,
        \begin{equation*}
            g(x) = \frac{1}{e^{ax} + e^{-bx}}, \quad \Re a, \Re b > 0. 
        \end{equation*}
        In particular, they solved the lattice case problem for another class of windows.
        These solutions have the form: $\GG$ is a frame whenever $D^-(\L) > \b$, where $D^-(\L)$ is \emph{the lower Beurling density}
        \begin{equation*}
            D^-(\L) = \liminf_{R \to \infty} \inf_{t \in \R} \frac{|\L \cap [t, t + R]|}{R}.
        \end{equation*}
        The main tool there is \emph{the shift-invariant space}
        \begin{equation*}
            V^2(g) = \left\{ f(x) = \sum_{n \in \Z} c_n g(x - n): c \in \ell^2(\Z) \right\}.
        \end{equation*}
        We assume that $g$ satisfies $\sum_{k \in \Z} \max_{x \in [k, k+1]} |g(x)| < \infty$ and $\sum_{k \in \Z} |\hat g (\xi + k)|^2 > 0, \ \xi \in \R$, where $\hat g$ is the Fourier transform of $g$. Under these two relatively weak conditions, we state the main theorem that relates Gabor frames to the shift-invariant space \cite[Theorem 2.3]{ShSGabor}.
        \begin{thrmm} \label{thrmm: frames to sampling}
            Suppose $\L \subset \R$ is separated. Then $\GG(g, \L \times \Z)$ is a frame if and only if $\L + x$ is sampling for $V^2(g)$ for any $x \in \R$.
        \end{thrmm}
        Recall that $\L \subset \R$ is said to be \emph{sampling} for $V^2(g)$ whenever there are $A, B > 0$ such that
        \begin{equation*}
            A \|f\|^2 \le \sum_{\l \in \L} |f(\l)|^2 \le B \|f\|^2, \quad f \in V^2(g).
        \end{equation*}
        In light of Theorem \ref{thrmm: frames to sampling} we are interested in sequences which remain sampling under shifts, that is, \emph{shift-invariant sampling} sequences. One of the key results in \cite{HypCh} is Lemma 5.1, which translates the problem from describing sampling sequences for $V^2(g)$ to describing positive sampling sequences for two-sided small Fock spaces. Thus, the main result of this paper provides an immediate and clear understanding underlying the result of Belov and Baranov \cite{HypCh}. In fact, the stated equivalence serves as one of the main motivations to study shift-invariant sampling sequences in two-sided small Fock spaces, and is the original motivation behind this paper.

    \subsection*{Notation} $A \lesssim B$ (equivalently $B \gtrsim A$) means that there is a constant $C > 0$ independent of the relevant variables such that $A \le CB$. We write $A \asymp B$ if both $A \lesssim B$ and $A \gtrsim B$.
    
    \subsection*{Organization of the paper} Section \ref{sect: prelim} contains preliminary results about two-sided small Fock spaces, which are needed for the rest of the paper. It sets the groundwork for the proof of Theorem \ref{thrm: ShS p = infty} in Section \ref{sect: Theorem 1}. Section \ref{sect: prelim p < inf} contains more preliminary results specific to the case $p < \infty$. 
    Finally, Section \ref{sect: Theorem 2} is where we finish by proving Theorem \ref{thrm: ShS p = 2}. 

\section{Preliminary results} \label{sect: prelim}
    In this section, we consider the symmetries of two-sided small Fock spaces, and next, we present the main tools of Beurling's balayage.
    \subsection{Shift-invariance of the spaces} \label{subsect: Shift-invariance}
    Two-sided small Fock spaces exhibit a kind of discrete shift-invariance, as stated in the following proposition.
    \begin{prop} \label{prop: shift operator}
        An operator $T_\a$ such that
        \begin{equation*}
            g(z) = (T_\a f)(z) = e^{-\frac{1}{4 \a}} z^{-1} f(e^{ \frac{1}{2\a}} z) 
        \end{equation*}
        satisfies
        \begin{equation*}
            |g(z)|e^{-\f(z)} = |f(w)|e^{-\f(w)}, \quad w = e^{\frac{1}{2\a}} z.
        \end{equation*}
        In particular, $T_\a$ is an isometric automorphism of $\F^{\infty}_\a$, while $e^{\frac{1}{p\a}} T_\a$ is an isometric automorphism of $\F^{p}_\a$ for any $0 < p < \infty$. 
    \end{prop}
    \begin{proof}
        It is sufficient to prove
        \begin{equation*}
            |g(z)|e^{-\f(z)} = |f(w)|e^{-\f(w)}, \quad w = e^{\frac{1}{2\a}} z.
        \end{equation*}
        Indeed, we have
        \begin{align*}
            |g(z)|e^{-\f(z)} &=  e^{-\frac{1}{4 \a}} |z|^{-1} |f(e^{ \frac{1}{2\a}} z)| e^{-\a \log^2 |z|} \\
            &= e^{-\frac{1}{4 \a}} e^{\frac{1}{2\a}} |w|^{-1} |f(w)| e^{-\a \left( \log|w| - \frac{1}{2\a} \right)^2} \\
            &= e^{\frac{1}{4 \a}} |w|^{-1} |f(w)| e^{-\a \log^2|w|} e^{\log|w|} e^{-\frac{1}{4\a}} = |f(w)|e^{-\f(w)}.
        \end{align*}
    \end{proof}
    \begin{crl}
        For any $n \in \Z$ the operator $T_\a^n$ satisfies that for $g = T_\a^n f$ we have
        \begin{equation*}
            |g(z)|e^{-\f(z)} = |f(w)|e^{-\f(w)}, \quad w = e^{\frac{n}{2\a}} z.   
        \end{equation*}
        Hence, $T_\a^n$ is an isometric automorphism of $\F^{\infty}_\a$ and $e^{\frac{n}{p\a}} T^n_\a$ is an isometric automorphism of $\F^{p}_\a$ for any $0 < p < \infty$.
    \end{crl}
    \begin{rmk}
        We have the explicit expression
        \begin{equation*}
            (T^n_{\a}f)(z) = e^{-\frac{n^2}{4 \a}} z^{-n} f(e^{\frac{n}{2\a}}z), \quad n \in \Z.
        \end{equation*}
    \end{rmk}
    \begin{proof}
        Denote the expression on the right by $S^nf$. It is easy to see that $T_\a S^n = S^{n+1}, \ n \in \Z$. Thus, $T^n_{\a} = S^n, \ n \in \Z$ follows from the fact that $S^0 = \id = T^0_{\a}$.
    \end{proof}
    \begin{rmk} \label{rmk: ShS is SS for any c}
        It follows from the existence of $T^n_{\a}$ that for any $\L$ the sampling constants $K_{e^t\L}$, \eqref{eq: K Lambda definiton}, and $A_{e^t\L}$, $B_{e^t\L}$, \eqref{eq: A B Lambda definiton}, are $\frac{1}{2\a}$-periodic. Together with the continuity result of the type \cite[Theorem 2, p.344]{Beurling} this periodicity implies that $\L$ being ShS is equivalent to $c \L$ being SS for all $c > 0$.
    \end{rmk}
    \begin{crl} \label{crl: zn norm estimate}
        For a fixed $0 < p \le \infty$
        \begin{equation*}
            \|z^n\|_{\F^p_{\a}} \asymp e^{\frac{(n + 2/p)^2}{4\a}}, \quad n \in \Z.
        \end{equation*}
    \end{crl}
    \begin{proof}
        Applying the isometry $e^{\frac{-n}{p\a}} T^{-n}_\a$ to $f = 1$ we get $e^{\frac{-n}{p\a}} e^{-\frac{n^2}{4\a}}z^n$. Hence,
        \begin{equation*}
            \|1\|_{\F^p_{\a}} = \|e^{\frac{-n}{p\a}} e^{-\frac{n^2}{4\a}}z^n\|_{\F^p_{\a}},
        \end{equation*}
        so that
        \begin{equation*}
            \|z^n\|_{\F^p_{\a}} \asymp e^{\frac{n}{p\a}} e^{\frac{n^2}{4\a}} \asymp e^{\frac{(n + 2/p)^2}{4\a}}, \quad n \in \Z.
        \end{equation*}
    \end{proof}

    \subsection{Beurling's balayage}
    Now, let us introduce the main concepts that are involved in Beurling's balayage technique. 
    \subsubsection{Weak limits of sequences}
    Given two finite sets $A$ and $B$ in $\C_0$ we say that $B$ is \emph{an $\e$-perturbation} of $A$ if there is a one-to-one correspondence between $A$ and $B$ such that the $d_{\log}$ distance between each point of $A$ and its image in $B$ does not exceed $\e$.
    \begin{definition}
        Suppose $\L_n \subset \C_0, \ n \in \N$ are separated with $d_{\L_n} \ge d$. We say that $\L_n$ \emph{weakly converge to} $\L$ if for any $\e > 0$ and any $a, b > 0$ there exists $N$ such that for $n > N$ the set 
        \begin{equation*}
            \{ \l \in \L: a < |\l| < b \}
        \end{equation*}
        is an $\e$-perturbation of
        \begin{equation*}
            \{ \l \in \L_n: a < |\l| < b \}.
        \end{equation*}
    \end{definition}
    In this case one easily verifies that
    \begin{equation*}
        d_{\L} \ge \limsup_{n \to \infty} d_{\L_n}.
    \end{equation*}
    For a set $\L \subset \C_0$, we denote by $R_{\L} = R: \F^{\infty}_\a \to \ell^{\infty}(\L)$ the restriction operator
    \begin{equation*}
        R: f \mapsto (f(\l)e^{-\a \log^2|\l|})_{\l \in \L}.
    \end{equation*}
    Note that $f \mapsto fe^{-\f}$ is an embedding that makes $\F^{\infty}_\a$ a closed subspace of $L^{\infty}(\C_0)$. Hence,
    \begin{equation*}
        \F^{\infty}_\a = \left( L^1(\C_0)/(\F^{\infty}_\a)^\perp \right)^*.
    \end{equation*}
    We denote $L^1(\C_0)/(\F^{\infty}_\a)^\perp$ by $X$. So, for any $h \in X$, the action of $f \in \F^{\infty}_\a$ on $h$ is
    \begin{equation*}
        (h, f) = \int_{\C_0} h(z)f(z)e^{-\f(z)} dA(z), 
    \end{equation*}
    where $h$ denotes any representative of its equivalence class.
        
    We claim that for any $w \in \C_0$ there is a $\k_w \in X$ such that
    \begin{equation*}
        (\k_w, g) = g(w)e^{-\f(w)}, \quad  g \in \F^{\infty}_\a.   
    \end{equation*}
    Indeed, we can take an $L^1(\C_0)$ function
    \begin{equation*}
        h(z) = \frac{1}{\pi \e^2} \chi_{w + \e \D}(z) e^{\f(z) - \f(w)},   
    \end{equation*}
    where $\e < |w|$. Then, for any $f \in \F^{\infty}_\a$
    \begin{equation*}
        (h, f) = \frac{e^{-\f(w)}}{\pi \e^2} \int_{w + \e \D} f(z) dA(z)  = f(w) e^{-\f(w)}.
    \end{equation*}
    So, the equivalence class of $h$ is the desired $\k_w$. Note that
    \begin{equation*}
        \| \k_w \|_X = \sup_{\| g \|_{\F^{\infty}_{\a}} \le 1 } |(\k_w, g)| = \sup_{\| g \|_{\F^{\infty}_{\a}} \le 1 } |g(w)|e^{-\f(w)} \le 1.
    \end{equation*}
    Hence, a map $T_\L = T: \ell^1(\L) \to X$, 
    \begin{equation*}
        T: (c_{\l})_{\l \in \L} \mapsto \sum_{\l \in \L} c_{\l} \k_{\l},
    \end{equation*}
    is well defined. It is easy to see that $R = T^*$.

    Let us restate the sampling property of $\L$ using this operator $T$.
    \begin{prop} \label{prop: T^* = R equivalence}
        $\L$ is SS for $\F^{\infty}_\a$ with $K_{\L} \le K$ if and only if for any $h \in X, \ \| h \|_X \le 1$ there is $c \in \ell^1(\L), \ \| c \|_1 \le K$ such that $Tc = h$.
    \end{prop}
    \begin{proof}
        The proposition follows from a general theorem by Banach.
        \begin{thrm*}
            Suppose $Y_1, Y_2$ are Banach spaces and $A:Y_1 \to Y_2$ is a bounded operator. Then the following are equivalent:
            \begin{itemize}
                \item $K > 0$ is such that for any $y_2 \in Y_2$ there is $y_1 \in Y_1: \|y_1\| \le K \|y_2\|$ and $Ay_1 = y_2$,
                \item $K\|A^* y^*_2\| \ge \|y^*_2\|$, for any $y^*_2 \in Y^*_2$.
            \end{itemize}    
        \end{thrm*}
        It remains to apply this theorem with $Y_1 = \ell^1(\L)$, $Y_2 = X$ and $A = T$.
    \end{proof}
    One of the key properties of weak convergence is the conservation of the sampling property.
    \begin{lem} \label{lem: weak limit}
        Suppose $\L_n \subset \C_0$, $d_{\L_n} \ge d$ weakly converge to $\L$ and $\L_n$ are SS for $\F^{\infty}_\a$ with $K_{\L_n} \le K$ for some $K > 0$ and all $n \in \N$. Then $\L$ is SS for $\F^{\infty}_\a$ and $K_\L \le K$.
    \end{lem}
    \begin{proof}
    By Proposition \ref{prop: T^* = R equivalence} it is enough to show that given any $h \in X, \ \| h \|_X \le 1$ there is $c \in \ell^1(\L), \ \|c\|_1 \le K$ such that $T_{\L} c = h$. Fix such $h \in X$. By Proposition \ref{prop: T^* = R equivalence} we know that there exists $c^{(n)} \in \ell^1(\L_n), \ \|c^{(n)}\|_1 \le K$ such that $T_{\L_n} c^{(n)} = h$. This means that for any $f \in \F^{\infty}_\a$
    \begin{align*}
        (T_{\L_n} c^{(n)}, f) = (h, f) & \iff (c^{(n)}, R_{\L_n} f) = (h, f) \\
        & \iff \sum_{\l^{(n)} \in \L_n} f(\l^{(n)}) e^{-\f(\l^{(n)})} c^{(n)}_{\l^{(n)}} = (h, f) \\
        & \iff \int_{\C_0} f(z)e^{-\f(z)} d\mu_n(z) = (h, f),
    \end{align*}
    where 
    \begin{equation*}
        \mu_n = \sum_{\l^{(n)} \in \L_n} c^{(n)}_{\l^{(n)}} \de_{\l^{(n)}}.
    \end{equation*}
    
    Consider $\M(\C_0)$, the space of all regular finite complex Borel measures on $\C_0$. $\M(\C_0) = (C_0(\C_0))^*$, where $C_0(\C_0)$ is the space of continuous functions on $\C_0$ vanishing at $0$ and $\infty$. 
    Note that $\mu_n \in \M(\C_0)$ and $\|\mu_n\| = \|c^{(n)}\|_1 \le K$. So, by passing to a subsequence we can assume that $\mu_n \xrightarrow{w^*} \mu$. First, $\|\mu\| \le \liminf_{n \to \infty} \|\mu_n\| \le K$. Second, $\supp \mu \subset \L$, since for any $f \in C_0(\C_0)$ with a compact support such that $\supp f \subset \C_0 \setminus \L$ we have $(f, \mu) = \lim (f, \mu_n) = 0$,  as $(f, \mu_n) = 0$ for large $n$, since $\supp f$ is compact and $\L_n \to \L$ weakly. We conclude that 
    \begin{equation*}
        \mu = \sum_{\l \in \L} c_\l \de_{\l}, \quad \|\mu\| = \|c\|_1.    
    \end{equation*}
    We want to show that $T_{\L} c = h$. As before, it is equivalent to showing that for any $f \in \F^{\infty}_\a$
    \begin{equation*}
        (fe^{-\f}, \mu) = (h, f).
    \end{equation*}
    We know that for all $n$
    \begin{equation*}
        (fe^{-\f}, \mu_n) = (h, f).
    \end{equation*}
    If $fe^{-\f} \in C_0(\C_0)$, then we get the desired equality by passing to the limit $n \to \infty$ via the $w^*$-convergence. In particular, if $f$ does not have essential singularities at $0$ and $\infty$, then 
    \begin{equation*}
        (fe^{-\f}, \mu) = (h, f)    
    \end{equation*}
    holds.
    
    Now suppose that $f$ has an essential singularity at $0$ and $fe^{-\f}$ vanishes at $\infty$. Then there exists a sequence $s_j \in \C_0$ such that $s_j \to 0$ and $f(s_j) \to 0$. We consider
    \begin{equation*}
        g_j(z) = z\frac{f(z) - f(s_j)}{z - s_j}. 
    \end{equation*}
    Note that $g_j e^{-\f}$ vanishes at $0$ and $\infty$, hence
    \begin{equation*}
        \int_{\C_0} g_j(z)e^{-\f(z)} d \mu(z) = \int_{\C_0} g_j(z)e^{-\f(z)} h(z) dA(z).
    \end{equation*}
    Applying $T^n_{\a}$ and using the maximum principle, it is easy to see that $\|g_j\|_{\F^{\infty}_{\a}} \lesssim \|f\|_{\F^{\infty}_{\a}}$. Thus, since $g_j(z) \to f(z)$ pointwise, we can take the limit under the integral to get
    \begin{equation*}
        \int_{\C_0} f(z)e^{-\f(z)} d \mu(z) = \int_{\C_0} f(z)e^{-\f(z)} h(z) dA(z).
    \end{equation*}
    For the remaining case, suppose that $f$ has an essential singularity at $\infty$. Then there exists a sequence $t_j \in \C_0$ such that $t_j \to \infty$ and $f(t_j) \to 0$. We consider
    \begin{equation*}
        f_j(z) = t_j \frac{f(z) - f(t_j)}{t_j - z}.
    \end{equation*}
    $f_j e^{-\f}$ vanish at $\infty$, so by the previous cases
    \begin{equation*}
        \int_{\C_0} f_j(z)e^{-\f(z)} d \mu(z) = \int_{\C_0} f_j(z)e^{-\f(z)} h(z) d A(z).
    \end{equation*}
    Similarly to before, by applying $T^n_{\a}$ and using the maximum principle, it is easy to see that $\|f_j\|_{\F^{\infty}_{\a}} \lesssim \|f\|_{\F^{\infty}_{\a}}$. Thus, since $f_j(z) \to f(z)$ pointwise, we can take the limit under the integral to get
    \begin{equation*}
        \int_{\C_0} f(z)e^{-\f(z)} d \mu(z) = \int_{\C_0} f(z)e^{-\f(z)} h(z) d A(z).
    \end{equation*}
    \end{proof}
    Denote $W(\L)$ the set of all weak limits of shifts $c\L, \ c > 0$. Note that any weak limit of sets in $W(\L)$ also belongs to $W(\L)$ and for any $\tilde \L \in W(\L)$ we have $D^-_{\log}(\tilde \L) \ge D^-_{\log}(\L)$.
    \begin{crl}
        If separated $\L$ is ShS for $\F^{\infty}_\a$ with a uniform sampling constant $K > 0$, then any $\tilde \L \in W(\L)$ is SS for $\F^{\infty}_\a$ with the same sampling constant $K$. In particular, $\tilde \L$ is ShS for $\F^{\infty}_\a$ with a uniform sampling constant $K$.
    \end{crl}
    \subsubsection{Uniqueness}
    We recall that $\L$ is \emph{a uniqueness set} for a set of functions $\F$ if $f(\l) = 0, \ \l \in \L$ implies $f = 0$ for any function $f \in \F$.
    The first key ingredient of Beurling's balayage is the ability to restate the sampling property into a statement that only involves the uniqueness property with the cost of considering all weak limits of a sequence. 
    \begin{lem} \label{lem: weak uniqueness}
        A separated $\L$ is ShS for $\F^{\infty}_\a$ if and only if any $\tilde \L \in W(\L)$ is a uniqueness set for $\F^{\infty}_\a$.
    \end{lem}
    \begin{proof}
        If $\L$ is ShS for $\F^{\infty}_\a$, then by Lemma \ref{lem: weak limit} any $\tilde \L \in W(\L)$ is SS for $\F^{\infty}_\a$. In particular, it is a uniqueness set.
        
        Suppose $\L$ is not ShS for $\F^{\infty}_\a$. This means that there exist sequences $c_j > 0$ and $f_j \in \F^{\infty}_\a$ such that $1 \le \| f_j \|_{\F^{\infty}_{\a}} < 2$ and $\| f_j |_{c_j \L} \|_{\infty, \a} \to 0$. Inequality $\| f_j \|_{\F^{\infty}_{\a}} \ge 1$ implies, in particular, that there is a $z_j \in \C_0$ such that $|f_j(z_j)|e^{-\f(z_j)} \ge 1/2$. Now we uniquely choose such $n_j \in \Z$ that 
        \begin{equation*}
            e^{\frac{n_j}{2\a}} \le |z_j| < e^{\frac{n_j + 1}{2\a}}.
        \end{equation*}
        Then we consider $g_j = T^{n_j}_\a f_j$. It satisfies $1 \le \| g_j \|_{\F^{\infty}_{\a}} < 2$, $\| g_j |_{c'_j \L} \|_{\infty, \a} = \| f_j |_{c_j \L} \|_{\infty, \a} \to 0$, where $c'_j = e^{-\frac{n_j}{2\a}} c_j$, and
        \begin{equation*}
            |g_j(w_j)|e^{-\f(w_j)} \ge 1/2, \quad 1 \le |w_j| < e^{\frac{1}{2\a}},
        \end{equation*}
        where $w_j = e^{-\frac{n_j}{2\a}} z_j$. By passing to subsequences, we can assume 
        \begin{equation*}
            w_j \to w, \qquad c'_j \L \xrightarrow{w} \tilde \L, \qquad g_j \xrightarrow{u.c.c.} g,   
        \end{equation*}
        where $1 \le |w| \le e^{\frac{1}{2\a}}$, $\tilde \L \in W(\L)$ and $\| g \|_{\F^{\infty}_{\a}} \le 2$. 
        
        Since $g_j \xrightarrow{u.c.c.} g$ and $w_j \to w$, then $g_j(w_j) \to g(w)$, hence $|g(w)|e^{-\f(w)} = \lim_{j \to \infty} |g(w_j)|e^{-\f(w_j)} \ge 1/2$, so $g \ne 0$. Similarly, since $g_j \xrightarrow{u.c.c.} g$ and $c'_j \L \xrightarrow{w} \tilde \L$, then for any $\tilde \l \in \tilde \L$ $g(\tilde \l) = \lim_{j \to \infty} g_j(\l_j)$, where $\l_j \in c'_j\L$ and $\l_j \to \tilde \l$. So $\| g_j |_{c'_j \L} \|_{\infty, \a} \to 0$ implies $g(\tilde \l) = 0$. We conclude that $g|_{\tilde \L} = 0$, and so $\tilde \L \in W(\L)$ is not a uniqueness set for $\F^{\infty}_\a$. 
    \end{proof}
    \subsubsection{Scale of spaces}
    Another important element of Beurling's technique is the existence of an entire scale of function spaces together with the fact that a sequence stays sampling for sufficiently close spaces in this scale.
    \begin{lem} \label{lem: extending sampling}
        Suppose that separated $\L$ is ShS for $\F^{\infty}_\a$. Then there exists $\e_0 > 0$ such that $\L$ is ShS for any $\F^{\infty}_{\a + \e}$ with $\e < \e_0$.
    \end{lem}
    \begin{proof}
        Suppose that this is not the case. This means that there exists a sequence $\de_j > 0, \ \de_j \to 0$, such that $\L$ is not ShS for $\F^{\infty}_{\a + \de_j}$. By Lemma \ref{lem: weak uniqueness} it implies that for each $j$ there exists $\L_j \in W(\L)$ and $f_j \in \F^{\infty}_{\a + \de_j}, \ \| f_j \|_{\F^{\infty}_{\a + \de_j}} = 1$ such that $f_j|_{\L_j} = 0$. Inequality $\| f_j \|_{\F^{\infty}_{\a + \de_j}} \ge 1$ implies, in particular, that there is a $z_j \in \C_0$ such that $|f_j(z_j)|e^{-\f_{\a + \de_j}(z_j)} \ge 1/2$. Now we uniquely choose such $n_j \in \Z$ that 
        \begin{equation*}
            e^{\frac{n_j}{2(\a + \de_j)}} \le |z_j| < e^{\frac{n_j + 1}{2(\a + \de_j)}}.
        \end{equation*}
        Then we consider $g_j = T^{n_j}_{\a + \de_j} f_j$. It satisfies $\| g_j \|_{\F^{\infty}_{\a + \de_j}} = 1$, $g_j|_{\L'_j} = 0$, where $\L'_j = e^{-\frac{n_j}{2\a}} \L_j \in W(\L)$, and
        \begin{equation*}
            |g_j(w_j)|e^{-\f(w_j)} \ge 1/2, \quad 1 \le |w_j| < e^{\frac{1}{2(\a + \de_j)}} < e^{\frac{1}{2\a}},
        \end{equation*}
        where $w_j = e^{-\frac{n_j}{2(\a + \de_j)}} z_j$. By passing to subsequences, we can assume 
        \begin{equation*}
            w_j \to w, \qquad \L'_j \xrightarrow{w} \tilde \L, \qquad g_j \xrightarrow{u.c.c.} g,   
        \end{equation*}
        where $1 \le |w| \le e^{\frac{1}{2\a}}$, $\tilde \L \in W(\L)$ and $\| g \|_{\F^{\infty}_{\a}} \le 1$. Since $g_j \xrightarrow{u.c.c.} g$ and $w_j \to w$, then $g_j(w_j) \to g(w)$, hence $|g(w)|e^{-\f(w)} = \lim_{j \to \infty} |g(w_j)|e^{-\f(w_j)} \ge 1/2$, so $g \ne 0$. Similarly, since $g_j \xrightarrow{u.c.c.} g$ and $\L'_j \xrightarrow{w} \tilde \L$, then for any $\tilde \l \in \tilde \L$ $g(\tilde \l) = \lim_{j \to \infty} g_j(\l_j)$, where $\l_j \in \L'_j$ and $\l_j \to \tilde \l$. So $g_j|_{\L'_j} = 0$ implies $g(\tilde \l) = 0$. Thus, $g|_{\tilde \L} = 0$, and so $\tilde \L \in W(\L)$ is not a uniqueness set for $\F^{\infty}_\a$. By Lemma \ref{lem: weak uniqueness}, $\L$ is not ShS for $\F^{\infty}_\a$, a contradiction.
    \end{proof}
    \subsubsection{Density of zero sets} 
    Finally, the last part of Beurling's technique is a density estimate on zero sets. It, in turn, provides one with some understanding of the behavior of uniqueness sets. 
    \begin{lem} \label{lem: Zero set}
        Suppose $\L$ is the set of zeros of a non-zero function $f \in \F^{\infty}_\a$. Then \mbox{$D^-_{\log}(\L) \le 2 \a$}.
    \end{lem}
    \begin{proof}
        Suppose, for a contradiction, that there is a function $f \in \F^{\infty}_\a$, such that its set of zeros $\L$ has the lower density $D^-_{\log}(\L) > 2\a$. Considering $f(e^{i \theta} z)$ instead of $f(z)$, we can assume that $f(1) \ne 0$. We can also assume $\|f\|_{\F^{\infty}_{\a}} = 1$. Put
        \begin{equation*}
            g(z) = f(e^z),
        \end{equation*}
        this is an entire function with $g(0) \ne 0$. The Jensen's formula for the circle of radius $R$ gives us
        \begin{equation} \label{eq: Jensen 1}
            \int_{0}^{R} \frac{n(t)}{t} dt = \frac 1 {2\pi} \int_{0}^{2\pi} \log|g(Re^{i \t})|dt - \log |g(0)|,  
        \end{equation}
        where $n(t)$ is the number of zeros of $g$ within $t \D$. Note that 
        \begin{equation*}
            |g(z)| = |f(e^z)| \le e^{\f(e^z)} = e^{\a \log^2|e^z|} = e^{\a x^2},
        \end{equation*}
        where $z = x + iy$. Hence, $ \log |g(z)| \le \a x^2$, and thus
        \begin{equation} \label{eq: Jensen 2}
            \frac 1 {2\pi} \int_{0}^{2\pi} \log|g(Re^{i \t})|dt \le \frac 1 {2\pi} \int_{0}^{2\pi} \a R^2 \cos^2 t dt = \frac{\a R^2}{2}.   
        \end{equation}
        Now, let us estimate $n(t)$. Since $D^-_{\log}(\L) > 2\a$, for a small enough $\e > 0$ there exists $l > 0$ such that on any set $(x, x + l) \times [-\pi, \pi)$ there are at least $(2\a + \e)l$ zeros. First, take $t = Nl$ for a natural number $N > 0$. Divide $t\D$ into $2N$ strips $t \D \cap (kl, kl + l) \times \R, \ k = -N, \ldots, = N - 1$. Since $g$ is $2 \pi i$ periodic, in the $k^{\text{th}}$ strip there are at least 
        \begin{equation*}
            (2\a + \e)l \left\lfloor \frac{2\sqrt{t^2 - l^2\max(k^2, (k+1)^2)}}{2 \pi} \right\rfloor
        \end{equation*}
        zeros. Thus, we can estimate $n(t)$ from below as
        \begin{align*}
            n(t) &\ge (2\a + \e)l \sum_{k = -N}^{N - 1} \left\lfloor \frac{\sqrt{t^2 - l^2\max(k^2, (k+1)^2)}}{\pi} \right\rfloor \\
            &= 2(2\a + \e)l \sum_{k = 0}^{N - 1} \left\lfloor \frac{\sqrt{t^2 - l^2(k+1)^2}}{\pi} \right\rfloor \\
            & \ge 2(2\a + \e)l \left( \sum_{k = 0}^{N - 1} \frac{\sqrt{t^2 - l^2(k+1)^2}}{\pi} - N \right) \\
            &= 2(2\a + \e)l \left( \frac{t}{\pi} \sum_{k = 0}^{N - 1} \sqrt{1 - \frac{(k+1)^2}{N^2} } - N \right) \\
            &= 2(2\a + \e)l \left( \frac{t}{\pi} \sum_{k = 0}^{N} \sqrt{1 - \frac{k^2}{N^2} } - \frac{t}{\pi } - N \right). 
        \end{align*}
        Estimating $\sum_{k = 0}^{N} \sqrt{1 - \frac{k^2}{N^2}} \frac{1}{N}$ from below by the integral, we get
        \begin{equation*}
            \sum_{k = 0}^{N} \sqrt{1 - \frac{k^2}{N^2}} \frac{1}{N} \ge \frac{\pi}{4}.    
        \end{equation*}
        Hence,
        \begin{align*}
            n(t) & \ge 2(2\a + \e)l \left( \frac{t}{\pi} \frac{N \pi}{4} - \frac{t}{\pi } - N \right) \\
            & = (2\a + \e) \left( \frac{t^2}{2} - t\left( \frac{2l}{\pi } + 2 \right) \right).
        \end{align*}
        Thus, 
        \begin{equation*}
            n(t) \ge (2\a + \e) \left( \frac{t^2}{2} - Ct \right),
        \end{equation*}
        for $t = Nl$ and some $C > 0$. Then for an arbitrary $t \ge l$ we have $Nl \le t < (N + 1)l$ and
        \begin{align*}
            n(t) \ge n(Nl) & \ge (2\a + \e) \left( \frac{(Nl)^2}{2} - CNl \right) \\
            & = (2\a + \e) \left( \frac{t^2}{2} - Ct - \left( C(t - Nl) + \frac{t^2 - (Nl)^2}{2} \right) \right) \\
            & \ge (2\a + \e) \left( \frac{t^2}{2} - Ct - \left( Cl + lt \right) \right).
        \end{align*}
        Thus, making $C$ larger, we can guarantee
        \begin{equation*}
            n(t) \ge (2\a + \e) \left( \frac{t^2}{2} - Ct \right), \quad t \ge l.
        \end{equation*}
        Finally, since $n(0) = 0$, making $C$ larger again, we get
        \begin{equation*}
            n(t) \ge (2\a + \e) \left( \frac{t^2}{2} - Ct \right), \quad t \ge 0.
        \end{equation*}
        Hence,
        \begin{equation*}
            \frac{n(t)}{t} \ge (2\a + \e) \left( \frac{t}{2} - C \right), \quad t \ge 0.
        \end{equation*}
        So we get an estimate
        \begin{equation*}
            \int_{0}^R \frac{n(t)}{t} dt \ge (2\a + \e) \left( \frac{R^2}{4} - C R\right)
        \end{equation*}
        Combining this estimate with Jensen's formula \eqref{eq: Jensen 1} and the estimate \eqref{eq: Jensen 2}, we get
        \begin{equation*}
            (2\a + \e) \left( \frac{R^2}{4} - C R\right) \le \frac{\a R^2}{2} - \log |g(0)|.    
        \end{equation*}
        Finally, dividing by $R^2$ and taking the limit $R \to \infty$, we arrive at the inequality $2 \a + \e \le 2 \a$, a contradiction.
    \end{proof}

\section{Proof of Theorem \ref{thrm: ShS p = infty}} \label{sect: Theorem 1}
    \subsection{Sufficiency}        
        Suppose $D^-_{\log}(\L) > 2\a$. We want to prove that $\L$ is ShS for $\F^{\infty}_\a$. By Lemma \ref{lem: weak uniqueness}, it is enough to prove that any $\tilde \L \in W(\L)$ is a uniqueness set for $\F^{\infty}_\a$. By the properties of weak set convergence, $D^-_{\log}(\tilde \L) > 2\a$. Hence, by Lemma \ref{lem: Zero set} it is a uniqueness set for $\F^{\infty}_\a$.
    \subsection{Necessity} 
        Now suppose that $\L$ is ShS for $\F^{\infty}_\a$. By Lemma \ref{lem: extending sampling} $\L$ is ShS for $\F^{\infty}_{\a + \e}$ for some $\e > 0$. Thus, to finish the proof, it is enough to prove the following:
        \begin{prop} \label{prop: density estimate}
            Suppose $\L$ has $D^-_{\log}(\L) < 2 \a$, then $\L$ is not SS for $\F^{\infty}_\a$.
        \end{prop}
            \begin{clm}
                If for an interval $I$, the number of points of $\L$ satisfies $|\L \cap e^I \TT| = (1 - \de) 2 \a |I|$. Then 
                \begin{equation*}
                    K_{\L} \ge C e^{\frac{\a}{4}\left( |I| \de - \frac{3}{2\a} \right)^2},    
                \end{equation*} 
                where the constant $C$ depends only on $\L$ and $\a$.
            \end{clm}
            \subsubsection{Proof of Proposition \ref{prop: density estimate}}
            Now, if $D^-_{\log}(\L) < 2\a$, then there exists $\de_0 > 0$ and a sequence of intervals $I_n$ such that $|I_n| \to \infty$ and $|\L \cap e^{I_n} \TT| = (1 - \de_n) 2 \a |I_n|$ with $\de_n \ge \de_0$. Hence,
            \begin{equation*}
                K_{\L} \ge C e^{\frac{\a}{4} \left( |I_n| \de_n - \frac{3}{2\a} \right)^2} \to \infty.
            \end{equation*}
            This means that $\L$ is not SS for $\F^{\infty}_\a$. Thus, it is sufficient to prove the claim.

            \subsubsection{Reduction to even number of points}
            We proceed with the proof.
            If $\de |I| \le \frac{3}{2\a}$, then there is nothing to prove, provided we have chosen a sufficiently small $C$. So we assume $\de |I| > \frac{3}{2\a}$. If $|\L \cap e^{I} \TT|$ is odd, then we can change $\de$ to $\de' = \de - \frac{1}{2 \a |I|}$ so that
            \begin{equation*}
                |\L \cap e^I \TT| = 2 \a (1 - \de) |I|, \qquad |\L \cap e^I \TT| + 1 = 2 \a (1 - \de')|I|.
            \end{equation*}
            If we know 
            \begin{equation*}
                K_{\L} \ge C e^{\frac{\a}{4}\left( |I| \de - \frac{1}{\a} \right)^2}   
            \end{equation*}
            for even $|\L \cap e^I \TT|$, then, since adding a point $\tilde \l \in e^I \TT$ to $\L$ can only decrease $K_{\L}$, we can estimate
            \begin{equation*}
                K_{\L} \ge K_{\L \cup \{ \tilde \l \} } \ge C e^{\frac{\a}{4}\left( |I| \de' - \frac{1}{\a} \right)^2} \ge C e^{\frac{\a}{4}\left( |I| \de - \frac{3}{2\a} \right)^2}.
            \end{equation*}
            Thus, it suffices to consider even $|\L \cap e^I \TT|$ and prove
            \begin{equation*}
                K_{\L} \ge C e^{\frac{\a}{4}\left( |I| \de - \frac{1}{\a} \right)^2}.
            \end{equation*}

            \subsubsection{Example}
            We set $|\L \cap e^I \TT| = 2N$ and $|I| = 2R$.
            Applying $T_\a^n$, we can assume that the interval $I$ starts between $-R - \frac{1}{2\a}$ and $-R$.  
            We denote by $\l_{-N}, \ldots, \l_{-1}, \l_1, \ldots, \l_N$ the points of $\L \cap e^I \TT$. 
            The numeration is in the order of growth of the modulus. Consider
            \begin{equation*}
                f(z) = \prod_{l = 1}^N \left( 1 - \frac{\l_{-l}}{z} \right) \prod_{k = 1}^N \left(1 - \frac{z}{\l_k} \right).
            \end{equation*}
            This function serves the role of an example that provides the bound for the sampling constant $K_{\L}$.
            
            In the following analysis, we consider only the points $z$ uniformly separated from $\L$ in $d_{\log}$.
            
            For a $z$ such that $|\l_n| \le |z| \le |\l_{n+1}|, \ n \ge 1$ or $|z| \ge |\l_n|, \ n = N$ we have
            \begin{equation*}
                \log |f(z)| = \sum_{k = 1}^n \log \left| \frac{z}{\l_k} \right| + O(1),
            \end{equation*}
            see \cite[Lemma 1]{SmallFockCIS} for the detailed proof.
            
            Similarly, for $z$ such that $|\l_{-n - 1}| \le |z| \le |\l_{-n}|, \ n \ge 1$ or $|z| \le |\l_{-n}|, \ n = N$ we have
            \begin{equation*}
                \log |f(z)| = \sum_{k = 1}^n \log \left| \frac{\l_{-k}}{z} \right| + O(1),
            \end{equation*}
            For $|\l_{-1}| \le |z| \le |\l_1|$, we get
            \begin{equation*}
                \log |f(z)| = O(1).    
            \end{equation*}
            The bounds on $O(1)$ depend only on the separation constant of $\L$.

            \subsubsection{Estimate on \texorpdfstring{$\| f \|_{\F^{\infty}_{\a}}$}{||f||}}
            Denote $A_n = \prod_{k = 1}^n |\l_k|^{-1}$, $B_n = \prod_{k = 1}^n |\l_{-k}|$. We claim that 
            \begin{equation*}
                \| f \|_{\F^{\infty}_{\a}} \gtrsim \max(A_N e^{\frac{N^2}{4 \a}}, B_N e^{\frac{N^2}{4 \a}}).
            \end{equation*}
            To prove $\| f \|_{\F^{\infty}_{\a}} \gtrsim A_N e^{\frac{N^2}{4 \a}}$, we consider $r = |z| = e^{\frac{N}{2\a}}$ and $z$ uniformly separated from $\L$ in $d_{\log}$. Suppose there is $n \ge 1$ such that $|\l_n| \le |z| \le |\l_{n + 1}|$. Then we have
            \begin{align*}
                \| f \|_{\F^{\infty}_{\a}} & \ge |f(z)| e^{-\a (\frac{N}{2\a})^2} \\
                & \gtrsim A_n |z|^n e^{-\a (\frac{N}{2\a})^2} = A_n e^{\frac{Nn}{2\a} - \frac{N^2}{4\a}} = A_n e^{\frac{N^2}{4\a}} e^{- \frac{N(N-n)}{2\a}} \\
                &= A_N e^{\frac{N^2}{4\a}} \frac{A_n}{A_N} e^{- \frac{N(N-n)}{2\a}} \ge A_N e^{\frac{N^2}{4\a}}, 
            \end{align*}
            since
            \begin{equation*}
                \frac{A_n}{A_N} = \prod_{k = n + 1}^N |\l_k| \ge |z|^{N - n} = e^{\frac{N(N-n)}{2\a}}.
            \end{equation*}
            If instead $|\l_{-1}| \le |z| \le |\l_1|$, then the same inequalities give us the result, just with $n = 0$. And if $|\l_{-n-1}| \le |z| \le |\l_{-n}|$ for some $n \ge 1$, then
            \begin{align*}
                \| f \|_{\F^{\infty}_{\a}} & \ge |f(z)| e^{-\a (\frac{N}{2\a})^2} \\
                & \gtrsim B_n |z|^{-n} e^{-\a (\frac{N}{2\a})^2} = B_n e^{\frac{-Nn}{2\a} - \frac{N^2}{4\a}} = B_n e^{\frac{N^2}{4\a}} e^{- \frac{N(N +n)}{2\a}} \\
                & = A_N e^{\frac{N^2}{4\a}} \frac{B_n}{A_N} e^{- \frac{N(N+n)}{2\a}} \ge A_N e^{\frac{N^2}{4\a}},
            \end{align*}
            since
            \begin{equation*}
                \frac{B_n}{A_N} = \prod_{l = 1}^{n} |\l_{-l}| \prod_{k = 1}^N |\l_k| \ge |z|^{N + n} = e^{\frac{N(N+n)}{2\a}}.
            \end{equation*}
            To prove $\| f \|_{\F^{\infty}_{\a}} \gtrsim B_N e^{\frac{N^2}{4 \a}}$, we can apply a similar argument or just consider $f \left( \frac 1 z \right)$.

            \subsubsection{Estimate on \texorpdfstring{$ \| f|_{\L} \|_{\infty, \a}$}{f restricted to lambda}}
            It remains to estimate $\| f|_{\L} \|_{\infty, \a} = \sup_{\l \in \L} |f(\l)|e^{-\f(\l)}$. By definition $f = 0$ on $\l_{-N}, \ldots, \l_{-1}, \l_1, \ldots, \l_N$. Hence, we need to estimate only for $|\l| \ge |\l_N|$ and $|\l| \le |\l_{-N}|$. First, consider $|\l| \ge |\l_N|$. Then
            \begin{equation*}
                |f(\l)|e^{-\f(\l)} \lesssim A_N|\l|^N e^{-\a \log^2 |\l|}.
            \end{equation*}
            Suppose that we know $\log |\l| \ge \frac{N}{2 \a} (1 + t)$ for some $t \ge 0$. Then, since the function $r^N e^{-\a \log^2 r}$ decreases from $e^{\frac{N}{2\a}}$ to $\infty$, we can estimate
            \begin{align*}
                A_N|\l|^N e^{-\a \log^2 |\l|} & \lesssim A_N e^{N \frac{N}{2 \a}(1 + t)} e^{-\frac{N^2}{4 \a}(1+t)^2} \\
                & = A_N e^{\frac{N^2}{4 \a} \left( 1 - t^2 \right)} = A_N e^{\frac{N^2}{4 \a}} e^{-\frac{N^2t^2}{4\a}} \\
                & \lesssim e^{-\frac{N^2t^2}{4\a}} \| f \|_{\F^{\infty}_{\a}}.
            \end{align*}
            By our construction, we have $\log |\l| \ge R - \frac{1}{2 \a}$ and it remains to notice that
            \begin{equation*}
                R - \frac{1}{2 \a} = \frac{N}{2\a}(1 + t)    
            \end{equation*}
            with
            \begin{equation*}
                t = \frac{\de - \frac{1}{2 \a R}}{1 - \de} \ge 0.
            \end{equation*}
            Hence,
            \begin{equation*}
                Nt = 2 \a  \left( R \de - \frac{1}{2\a} \right).
            \end{equation*}
            Thus, we get
            \begin{equation*}
                |f(\l)|e^{-\f(\l)} \lesssim e^{-\a\left( R \de - \frac{1}{2\a} \right)^2} \| f \|_{\F^{\infty}_{\a}} = e^{-\frac{\a}{4}\left( |I| \de - \frac{1}{\a} \right)^2} \| f \|_{\F^{\infty}_{\a}}.
            \end{equation*}
            Now we consider $|\l| \le |\l_{-N}|$. Then,
            \begin{equation*}
                |f(\l)|e^{-\f(\l)} \lesssim B_N|\l|^{-N} e^{-\a \log^2 |\l|}.
            \end{equation*}
            Suppose that we know $\log |\l| \le -\frac{N}{2 \a} (1 + t)$ for some $t \ge 0$. Then, since the function $r^{-N} e^{-\a \log^2 r}$ increases from $0$ to $e^{\frac{-N}{2\a}}$, we can estimate
            \begin{align*}
                B_N|\l|^{-N} e^{-\a \log^2 |\l|} & \lesssim B_N e^{N \frac{N}{2 \a}(1 + t)} e^{-\frac{N^2}{4 \a}(1+t)^2} \\
                & = B_N e^{\frac{N^2}{4 \a} \left( 1 - t^2 \right)} = B_N e^{\frac{N^2}{4 \a}} e^{-\frac{N^2t^2}{4\a}} \\
                & \lesssim e^{-\frac{N^2t^2}{4\a}} \| f \|_{\F^{\infty}_{\a}}.
            \end{align*}
            In this case, by our construction, $\log|\l| \le -R$ and so we can take 
            \begin{equation*}
                t = \frac{\de}{1 - \de}
            \end{equation*}
            so that $R = \frac{N}{2 \a}(1 + t)$. Hence, we get
            \begin{equation*}
                Nt = 2 \a \de R.    
            \end{equation*}
            Thus,
            \begin{equation*}
                |f(\l)|e^{-\f(\l)} \lesssim e^{-\a\left( R \de \right)^2} \| f \|_{\F^{\infty}_{\a}} = e^{-\frac{\a}{4}\left( |I| \de \right)^2} \| f \|_{\F^{\infty}_{\a}} \le e^{-\frac{\a}{4}\left( |I| \de - \frac{1}{\a} \right)^2} \| f \|_{\F^{\infty}_{\a}}.    
            \end{equation*}
            In conclusion, we have found such an $f$ that
            \begin{equation*}
                \| f|_{\L} \| \lesssim e^{-\frac{\a}{4}\left( |I| \de - \frac{1}{\a} \right)^2} \| f \|_{\F^{\infty}_{\a}},
            \end{equation*}
            which implies the desired
            \begin{equation*}
                K_{\L} \ge C e^{\frac{\a}{4}\left( |I| \de - \frac{1}{\a} \right)^2}.
            \end{equation*}
            This completes the proof of the claim, and hence, the proof of the Theorem. \hfill \qedsymbol
    
\section{Preliminary results for \texorpdfstring{$p < \infty$}{p < ∞}} \label{sect: prelim p < inf}
    In this section, we prove some necessary properties of two-sided small Fock spaces for finite $p > 0$. Then, we obtain results concerning the behavior of sampling sequences that are necessary to pass the density description from $p = \infty$ to $0 < p < \infty$.    
    \subsection{Some properties of the spaces}
        \subsubsection{Lower bound}
        For finite $p$, in the context of sampling, we have to consider two constants $A, B$ in \eqref{eq: sampling definition finite p}. The boundedness of $B$ is essential to the sampling problem, while the existence of $A > 0$ is simply related to the separateness of $\L$.
        \begin{prop} \label{prop: separation to restriction}
            Suppose $\L$ is separated and $0 < p < \infty$. Then there exists $A > 0$ depending only on $d_{\L}$ and $p$ such that for any $f \in \F^p_\a$
            \begin{equation*}
                A \| f|_{\L} \|_{p, \a}^p \le \| f \|_{\F^p_{\a}}^p.    
            \end{equation*}
        \end{prop}        
        \begin{proof}
            Consider $a \in \C_0$, $1 \le |a| < e^{\frac{1}{2\a}}$ and $r \in (0, 1)$. Then for any $g \in \F^p_\a$ by subharmonicity  
            \begin{equation*}
                |g(a)|^p \le \frac{1}{\pi r^2} \int_{|w - a| < r} |g(w)|^p dA(w).
            \end{equation*}
            That implies the existence of a constant $C_r > 0$
            \begin{equation*}
                |g(a)|^p |a|^2 e^{-p\f(a)} \le C_r \int_{|w - a| < r} |g(w)|^p e^{-p\f(w)} dA(w). 
            \end{equation*}
            Now take any $f \in \F^p_\a$ and $a \in \C_0$. There exists a unique $n \in \Z$ such that $e^{\frac{n}{2\a}} \le |a| < e^{\frac{n + 1}{2\a}}$. Denote $g = T^n_\a f$ and $b = a e^{-\frac{n}{2\a}}$, $1 \le |b| < e^{\frac{1}{2\a}}$. For such $b$ we have
            \begin{equation*}
                |g(b)|^p |b|^2 e^{-p\f(b)} \le C_r \int_{|w - b| < r} |g(w)|^p e^{-p\f(w)} dA(w). 
            \end{equation*}
            By Proposition \ref{prop: shift operator} we have $|g(b)|e^{-\f(b)} = |f(a)|e^{-\f(a)}$. Hence,
            \begin{equation*}
                |f(a)|^p e^{-p\f(a)} |b|^2 \le C_r \int_{|w - b| < r} |g(w)|^p e^{-p\f(w)} dA(w).
            \end{equation*}
            Making a substitution $w = e^{-\frac{n}{2 \a}} z$, we get
            \begin{equation*}
                \int_{|w - b| < r} |g(w)|^p e^{-p\f(w)} dA(w) = \int_{|z - a| < r e^{\frac{n}{2\a}}} |f(z)|^p e^{-p\f(z)} e^{-\frac{2n}{2\a}} dA(z).
            \end{equation*}
            Thus, 
            \begin{equation*}
                |f(a)|^p |a|^2 e^{-p\f(a)} \le C_r \int_{|z - a| < r e^{\frac{n}{2\a}}} |f(z)|^p e^{-p\f(z)} dA(z),
            \end{equation*}
            which implies
            \begin{equation*}
                |f(a)|^p |a|^2 e^{-p\f(a)} \le C_r \int_{|z - a| < r |a|} |f(z)|^p e^{-p\f(z)} dA(z),
            \end{equation*}
            for any $f \in \F^p_\a$ and $a \in \C_0$. Hence, we get
            \begin{equation*}
                \| f|_{\L} \|_{p, \a}^p = \sum_{\l \in \L} |f(\l)|^p |\l|^2 e^{-p\f(\l)} \le C_r \sum_{\l \in \L} \int_{|z - \l| < r |\l|} |f(z)|^p e^{-p\f(z)} dA(z).   
            \end{equation*}
            Now, if $\L$ is separated with a constant of separation $d_{\L}$, one can choose a sufficiently small $r > 0$ depending only on $d_{\L}$ so that the sets $|z - \l| < r |\l|$ do not intersect for different $\l$. Hence,
            \begin{equation*}
                \| f|_{\L} \|_{p, \a}^p \le C_r \|f\|^p_{\F^p_{\a}}.
            \end{equation*}
            It remains to take $A = \frac{1}{C_r}$ to get the desired inequality.
        \end{proof}
        \begin{crl} \label{crl: eval bound}
            Take any $0 < p < \infty$. Considering $\L = \{ w \}$ in Proposition \ref{prop: separation to restriction} we get an estimate of the evaluation functional
            \begin{equation} \label{eq: evaluation estimate}
                |f(w)| \lesssim \frac{e^{\f(w)}}{|w|^{2/p}} \| f \|_{\F^p_{\a}}.
            \end{equation}
        \end{crl}
        In fact, by considering monomials, it can be seen that the factor in \eqref{eq: evaluation estimate} gives an equivalent estimate for the norm of the evaluation functional. In particular, that provides an explanation for the weights in \eqref{eq: restriction norm}.
        
        Henceforth, if we have sets with uniform separation constants, we only need to worry about the parameter $B$ in the sampling inequalities \eqref{eq: sampling definition finite p}.

        \subsubsection{Density of Laurent polynomials}
        When considering the sampling property of weak limits for $0 < p < \infty$, we will need the following.
        \begin{lem} \label{lem: Laurent pol density}
            Laurent polynomials are dense in $\F^p_{\a}$ for all $0 < p < \infty$.
        \end{lem}
        \begin{proof}
            By Theorem \ref{thrm: CIS characterisation} a sequence $\G = \{ \g_n = e^{\frac{n + 2/p - 1}{2\a}} \}_{n \in \Z}$ is complete interpolating for $\F^p_{\a}$. We set the canonical product
            \begin{equation*}
                G(z) = \prod_{m \ge 0} \left( 1 - \frac{\g_{-m}}{z} \right) \prod_{n \ge 1} \left( 1 - \frac{z}{\g_n} \right).  
            \end{equation*}
            and we get an estimate 
            \begin{equation*}
                |G(z)| \asymp d_{\log}(z, \G) |z|^{1/2} \frac{e^{\f(z)}}{|z|^{2/p}},
            \end{equation*}
            see \cite[Lemma 2]{SmallFockCIS}.
            It follows that
            \begin{equation*}
                g_n(z) = \frac{G(z)}{G'(\g_n)(z - \g_n)}, \ n \in \Z
            \end{equation*}
            belong to $\F^p_{\a}$. Moreover, $g_n(\g_m) = \de_{nm}$. Thus, since $\G$ is complete interpolating, for any $f \in \F^p_{\a}$ we have 
            \begin{equation*}
                f = \sum_{n \in \Z} f(\g_n)g_n,
            \end{equation*}
            with convergence in $\F^p_{\a}$. We conclude that it is enough to show that each $g_n, \ n \in \Z$ can be approximated by Laurent polynomials in $\F^p_{\a}$. We prove an even stronger claim.
            \begin{clm}
                The Laurent series of $g_n, \ n \in \Z$ converges to it in $\F^p_{\a}$.
            \end{clm}
            Put
            \begin{equation*}
                g_n(z) = \sum_{m \in \Z} a_m z^m, \ z \in \C.
            \end{equation*}
            For any $m \in \Z$ and $R > 0$ we have
            \begin{equation*}
                a_m = \frac{1}{2\pi R^m} \int_{-\pi}^{\pi} g_n(R e^{it}) e^{-imt} dt. 
            \end{equation*}
            Hence,
            \begin{equation*}
                |a_m| \le R^{-m} \max_{|z| = R} |g_n(z)|.
            \end{equation*}
            For $m \ge 0$ we consider $R_m = e^{\frac{m + 1/2 + 2/p}{2\a}}$. For $|z| \to \infty$ we have $|g_n(z)| \asymp |z|^{-1} |G(z)|$, therefore,
            \begin{equation*}
                |a_m| \lesssim R_m^{-m} R_m^{-1/2} \frac{e^{\f(R_m)}}{R_m^{2/p}} = e^{- \frac{(m + 1/2 + 2/p)^2}{4\a}}, \ m \ge 0.
            \end{equation*}
            Similarly, for $m < 0$ we consider $R_m = e^{\frac{m - 1/2 + 2/p}{2\a}}$. For $|z| \to 0$ we have $|g_n(z)| \asymp |G(z)|$, therefore,
            \begin{equation*}
                |a_m| \lesssim R_m^{-m} R_m^{1/2} \frac{e^{\f(R_m)}}{R_m^{2/p}} = e^{- \frac{(m - 1/2 + 2/p)^2}{4\a}}, \ m < 0.
            \end{equation*}
            Thus, in general, we can estimate
            \begin{equation*}
                |a_m| \lesssim e^{- \frac{(m + 2/p)^2 + |m|}{4\a}}, \ m \in \Z.
            \end{equation*}
            Together with the estimate on $\|z_n\|$, Corollary \ref{crl: zn norm estimate}, we get
            \begin{equation*}
                \|a_m z^m\|_{\F^p_{\a}} \lesssim e^{-\frac{|m|}{4 \a}}, \ m \in \Z.
            \end{equation*}
            We conclude that the Laurent series of $g_n$ converges absolutely in $\F^p_{\a}$ and thus converges to $g_n$ in $\F^p_{\a}$.
        \end{proof}  
    \subsection{Properties of sampling}
        \subsubsection{Weak limits of sequences}
        Similarly to the case $p = \infty$, the sampling property is conserved by weak limits for finite $p > 0$.
        \begin{lem} \label{lem: weak sampling p = 2}
            Fix $0 < p < \infty$. Suppose $\L_n \subset \C_0$ are separated with $d_{\L_n} \ge d$, $\L_n$ are SS for $\F^p_\a$ with $B_{\L_n} \le B$ and $\L_n$ converge weakly to $\L$. Then $\L$ is SS for $\F^p_\a$ with $B_{\L} \le B$.
        \end{lem}
        \begin{proof}
            We need to show that for any $f \in \F^p_\a$
            \begin{equation*}
                \| f \|_{\F^p_{\a}}^p \le B \| f|_\L \|^p_{p, \a}.
            \end{equation*}
            Since the restriction operator $f \mapsto f|_\L$ is bounded by Proposition \ref{prop: separation to restriction}, it is enough to show the desired inequality on a dense subset of $\F^p_{\a}$. By Lemma \ref{lem: Laurent pol density}, the Laurent polynomials are dense in $\F^p_{\a}, \ 0 < p < \infty$.
            Hence, we assume that $f$ is a Laurent polynomial. Let us enumerate $\L = ( \l_k )_{k \in \Z}$ so that $|\l_k| \le |\l_{k + 1}|$. We do the same for all $\L_n = ( \l^{(n)}_k )_{k \in \Z}$. Now, since $\L_n$ converge weakly to $\L$, by making an appropriate shift in enumeration, we can assume $\l^{(n)}_k \to \l_k, \ k \in \Z$. Let $N$ denote the minimal power of $z$ in $f$ and $M$ the maximal power. Then for $|z| \to \infty$ we have
            \begin{align*}
                |f(z)|^p |z|^2 e^{-p\f(z)} & \lesssim |z|^{pM + 2} e^{-p\f(z)} = e^{(pM + 2)\log|z| - p\a \log^2|z|} \\
                & \lesssim e^{-p\a \left( \log|z| - \frac{pM + 2}{2p\a} \right)^2}
            \end{align*}
            Similarly, for $|z| \to 0$,
            \begin{equation*}
                |f(z)|^p |z|^2 e^{-p\f(z)} \lesssim e^{-p\a \left( \log|z| - \frac{pN + 2}{2 p \a} \right)^2}.
            \end{equation*}
            Thus, since $\l^{(n)}_0 \to \l_0$ and all $\L_n$ are uniformly separated, there is a sufficiently small $\de > 0$ and a sufficiently large $K_0 > 0$ such that for $k \ge K_0$  
            \begin{equation*}
                |f(\l^{(n)}_k)|^p |\l^{(n)}_k|^2 e^{-p\f(\l^{(n)}_k)} \lesssim e^{-p\a \left( \de k - \frac{pM + 2}{2p\a} \right)^2},
            \end{equation*}
            and, similarly, for $k \le -K_0$
            \begin{equation*}
                |f(\l^{(n)}_k)|^p |\l^{(n)}_k|^2 e^{-p\f(\l^{(n)}_k)} \lesssim e^{-p\a \left( -\de k + \frac{pN + 2}{2 p \a} \right)^2}.
            \end{equation*}
            That means that for any $\e > 0$ we can choose $K \ge K_0$ such that
            \begin{equation*}
                \sum_{|k| > K} |f(\l^{(n)}_k)|^p |\l^{(n)}_k|^2 e^{-p\f(\l^{(n)}_k)} < \e, \quad n \in \N.
            \end{equation*}
            Since $B_{\L_n} \le B$, we have
            \begin{equation*}
                \frac{1}{B} \| f \|^p_{\F^p_{\a}} \le \sum_{k \in \Z} |f(\l^{(n)}_k)|^p |\l^{(n)}_k|^2 e^{-p\f(\l^{(n)}_k)} \le \sum_{|k| \le K} |f(\l^{(n)}_k)|^p |\l^{(n)}_k|^2  e^{-p\f(\l^{(n)}_k)} + \e.
            \end{equation*}
            We pass to the limit $n \to \infty$ in these finitely many points to get
            \begin{equation*}
                \frac{1}{B} \| f \|^p_{\F^p_{\a}} \le \sum_{|k| \le K} |f(\l_k)|^p |\l_k|^2  e^{-p\f(\l_k)} + \e .    
            \end{equation*}
            Hence,
            \begin{equation*}
                \frac{1}{B} \| f \|^p_{\F^p_{\a}} \le \| f|_{\L} \|^p_{p, \a} + \e, \quad \e > 0.
            \end{equation*}
            Taking $\e \to 0$, we get the desired
            \begin{equation*}
                \| f \|^p_{\F^p_{\a}} \le B \| f|_{\L} \|^p_{p, \a}.  
            \end{equation*}  
        \end{proof}
        \begin{crl}
            Let $0 < p < \infty$. If separated $\L$ is ShS for $\F^p_\a$, then any $\tilde \L \in W(\L)$ is ShS for $\F^p_\a$ as well.
        \end{crl}

        \subsubsection{Removing a point from a sampling sequence}
        We say that $\L$ is \emph{minimal} for $\F^p_\a$ if, for each $\tilde \l \in \L$, $\L \setminus \{ \tilde \l \}$ is not a uniqueness set for $\F^p_\a$. 
        
        Note that if $\L \setminus \{ \tilde \l_0 \}$ is not a uniqueness set for $\F^p_\a$ for some $\tilde \l_0 \in \L$, then $\L \setminus \{ \tilde \l \}$ is not a uniqueness set for $\F^p_\a$ for any $\tilde \l \in \L$, and hence, $\L$ is minimal for $\F^p_\a$. Indeed, if a non-zero $f \in \F^p_\a$ is such that $f|_{\L \setminus \{ \tilde \l_0 \}} = 0$, then $f_{\tilde \l}(z) = f(z) \frac{z - \tilde \l_0}{z - \tilde \l} \in \F^p_\a$ is a non-zero function such that $f_{\tilde \l}|_{\L \setminus \{ \tilde \l \}} = 0$.
        Similarly, note that being complete interpolating is the same as being sampling and minimal. 
        
        \begin{lem} \label{lem: sampling or minimal}
            Suppose that a separated $\L$ is SS for $\F^p_\a$. Then $\L$ is minimal or $\L \setminus \{\l' \}$ is SS for $\F^p_\a$, $\l' \in \L$.
        \end{lem}
        \begin{proof}
            Suppose $\L \setminus \{\l' \}$ is not SS for $\F^p_\a$, for some $\l' \in \L$. To prove that $\L$ is minimal for $\F^p_\a$, it suffices to show that $\L \setminus \{\l' \}$ is not a uniqueness set for $\F^p_\a$. Since $\L \setminus \{\l' \}$ is not SS for $\F^p_\a$, there exists $f_j \in \F^p_\a$ such that $\| f_j \|_{\F^p_{\a}} = 1$ and $\| f_j|_{\L \setminus \{\l' \}} \|_{p, \a} \to 0$. $\L$ is SS for $\F^p_\a$, so $\| f_j|_{\L} \|_{p, \a} \gtrsim 1$. Hence, $|f_j(\l')| \gtrsim 1$. Taking a subsequence, we can assume that $f_j \xrightarrow{u.c.c.} f$. Then $f \in \F^p_\a$ and $|f(\l')| = \lim_{j \to \infty} |f_j(\l')| > 0$, hence $f \ne 0$. At the same time $f_j(\l) \to 0$ for any $\l \in \L \setminus \{\l' \}$. Thus, $f|_{\L \setminus \{\l' \}} = 0$.
        \end{proof}

        The following lemma is the key element that allows us to pass the necessary density condition from Theorem \ref{thrm: ShS p = infty} to Theorem \ref{thrm: ShS p = 2}.
        
        \begin{lem} \label{lem: removing point sampling}
            Fix $0 < p < \infty$. Suppose $\L \subset \C_0$ is separated and for any $c > 0$, the set $c \L$ is SS for $\F^p_\a$. Then removing one point from $\L$ does not change this property, i.e., for any $\tilde \l \in \L$ and any $c > 0$ $c (\L \setminus \{ \tilde \l \})$ is SS for $\F^p_\a$.
        \end{lem}        
        \begin{rmk}
            It follows that one can remove finitely many points from $\L$ without changing this property.
        \end{rmk}

        This lemma is equivalent to the non-existence of shift-invariant complete interpolating sequences. It plays the same role as Lemma 2.1 in \cite{SeipBargFock}.
        
        \begin{proof}
            It is enough to prove that $\L \setminus \{ \tilde \l \}$ is SS for $\F^p_\a$, since one can always swap $\L$ for $c \L, \ c > 0$.
            
            Suppose $\L \setminus \{ \tilde \l \}$ is not SS for $\F^p_\a$. Then by Lemma \ref{lem: sampling or minimal} $\L$ is minimal and, hence, a complete interpolating sequence for $\F^p_\a$. We use Theorem \ref{thrm: CIS characterisation}. Then, writing $\L = ( \l_k )_{k \in \Z}$ with $|\l_k| \le |\l_{k + 1}|$ and shifting the numeration appropriately we have $\l_k = e^{\frac{k + 2/p - 1 + \de_k}{2\a}} e^{i \theta_k}$ with $(\de_k)_{k \in \Z} \in \ell^{\infty}(\Z)$ and there exist $N \in \N$ and $\de > 0$ such that
            \begin{equation*}
                \sup_{n \in \Z} \left| \frac{1}{N} \sum_{k = n + 1}^{n + N} \de_k \right| \le \de < \frac{1}{2}.
            \end{equation*}
            For a shift $e^{\frac{t}{2\a}} \L$ let us write $e^{\frac{t}{2\a}} \L = (e^{\frac{t}{2\a}} \l_k)_{k \in \Z}$, $e^{\frac{t}{2\a}} \l_k = e^{\frac{k + 2/p - 1 + \de_k + t}{2\a}} e^{i \theta_k}$. From Remark \ref{rmk: m shift}, $e^{\frac{t}{2\a}} \L$ is a complete interpolating sequence if there exists a unique $m = m(t)$, such that there are $M \in \N$ and $\de' > 0$ such that
            \begin{equation} \label{eq: averaging property}
                \sup_{n \in \Z} \left| \frac{1}{M} \sum_{k = n + 1}^{n + M} \de_k + t + m \right| \le \de' < \frac{1}{2}.
            \end{equation}
            Set
            \begin{equation*}
                t_0 = \sup \left\{t \in [0, 1]: e^{\frac{s}{2\a}} \L \text{ is complete interpolating with } m(s) = 0, \text{ for all } s \in [0, t] \right\}.
            \end{equation*}
            Clearly, $t_0 \ge 0$, since for $t = 0$ we have $m(0) = 0$. Moreover, for small enough $t > 0$ we have the property \eqref{eq: averaging property} with $m(t) = 0$. Thus, $t_0 > 0$. On the other hand, for $t = 1$ we have $m(1) = -1$, so, similarly, we have the property \eqref{eq: averaging property} with $m(t) = -1$ for $t < 1$ with $1 - t$ small enough, so $t_0 < 1$.
            
            We claim that $e^{\frac{t_0}{2\a}} \L$ is not a complete interpolating sequence for $\F^p_\a$. Suppose it is, then if $m(t_0) = 0$, then we have the property \eqref{eq: averaging property} with $m(t) = 0$ for $t > t_0$ with $t - t_0$ small enough, which contradicts the definition of $t_0$. If $m(t_0) \ne 0$, then we have the property \eqref{eq: averaging property} with $m(t) = m(t_0)$ for $t < t_0$ with $t_0 - t$ small enough, which again contradicts the definition of $t_0$.
            
            Since $e^{\frac{t_0}{2\a}} \L$ is SS for $\F^p_\a$ and not minimal, then, by Lemma \ref{lem: sampling or minimal}, $e^{\frac{t_0}{2\a}} \L \setminus \{ e^{\frac{t_0}{2\a}} \l_0 \}$ is SS for $\F^p_\a$. We claim that $e^{\frac{t_0}{2\a}} \L \setminus \{ e^{\frac{t_0}{2\a}} \l_0 \}$ is minimal for $\F^p_\a$. For any $\l_k \in \L, \ k \ne 0$ there is $f \in \F^p_\a$ such that $f(\l) = 0, \ \l \in \L, \ \l \ne \l_k$. Now consider
            \begin{equation*}
                g(z) = z \frac{f \left( e^{-\frac{t_0}{2\a}} z \right)}{z - e^{\frac{t_0}{2\a}}\l_0 }.
            \end{equation*}
            Clearly, $g(e^{\frac{t_0}{2\a}}\l) = 0$, $\l \in \L$ with $\l \ne \l_0, \l_k$. It remains to check that $g \in \F^p_\a$.
            \begin{equation*}
                \| g \|^p_{\F^p_{\a}} = \int_{\C_0} |g(z)|^p e^{-p \f(z)} dA(z).
            \end{equation*}
            Substituting $z = e^{\frac{t_0}{2\a}}w$ and noting 
            \begin{equation*}
                \f(z) = \a \log^2 |z| = \a \left( \log|w| + \frac{t_0}{2\a} \right)^2 = \f(w) + t_0 \log|w| + \frac{t_0^2}{4\a}
            \end{equation*}
            we get
            \begin{equation*}
                \| g \|^p_{\F^p_{\a}} \lesssim \int_{\C_0} \left| \frac{wf(w)}{w - \l_0} \right|^p e^{-p\f(w)} |w|^{-pt_0} dA(w).
            \end{equation*}
            It is enough to notice that for large $|w|$ we have
            \begin{equation*}
                \left| \frac{wf(w)}{w - \l_0} \right|^p e^{-p\f(w)} |w|^{-pt_0} \lesssim |f(w)|^p e^{-p\f(w)},
            \end{equation*}
            while for small $|w|$ we have
            \begin{equation*}
                \left| \frac{wf(w)}{w - \l_0} \right|^p e^{-p\f(w)} |w|^{-pt_0} \lesssim |w|^{p - pt_0} |f(w)|^p e^{-p\f(w)} \le |f(w)|^p e^{-p\f(w)}.
            \end{equation*}
            Thus, we proved that $e^{\frac{t_0}{2\a}} \L \setminus \{ e^{\frac{t_0}{2\a}} \l_0,  e^{\frac{t_0}{2\a}} \l_k \}$ is not a uniqueness set for $\F^p_\a$ for any $\l_k \in \L, \l_k \ne \l_0$. Hence, $e^{\frac{t_0}{2\a}} \L \setminus \{ e^{\frac{t_0}{2\a}} \l_0 \}$ is minimal and SS for $\F^p_\a$. So, it is a complete interpolating sequence for $\F^p_\a$. Then, by Theorem \ref{thrm: CIS characterisation}, $e^{\frac{t}{2\a}} \L \setminus \{ e^{\frac{t}{2\a}} \l_0 \}$ is also a complete interpolating sequence for $t < t_0$ provided $t_0 - t$ is small enough. But by the definition of $t_0$, $e^{\frac{t}{2\a}} \L$ is itself a complete interpolating sequence, a contradiction.
        \end{proof}

        \subsubsection{From infinite \texorpdfstring{$p$}{p} to finite \texorpdfstring{$p$}{p}} 
        To pass the sufficient density condition from $p = \infty$ to $0 < p < \infty$ we prove the following. 
        \begin{lem} \label{lem: alpha+eps, p = infty to alpha, p = 2}
            Let $0 <  p < \infty$. For any $\e > 0$ there exists $C_{\e} > 0$ such that if separated $\L$ is SS for $\F^{\infty}_{(1 + \e)\a}$ with $K_{\L} = K$, then $\L$ is SS for $\F^p_\a$ with $B_{\L} \le C_{\e} K^p$.
        \end{lem}
        \begin{proof}
            We use Beurling's duality argument \cite[pp.348--358]{Beurling}, see also \cite[Lemma 2.7]{SmallFock} and \cite[Theorem 35]{marco2003interpolating}.
        
            First, consider 
            \begin{equation*}
                \F^{\infty, 0}_{\f} = \{ f \in \F^{\infty}_{\f}: |f(z)|e^{-\f(z)} \to 0, \ z \to 0, \infty \}
            \end{equation*}
            and the respective two-sided small Fock spaces $\F^{\infty, 0}_{\a}$.
            
            Since $\L$ is SS for $\F^{\infty}_{(1 + \e)\a}$
            \begin{equation*}
                T_z: ( f(\l)e^{-(1 + \e)\a \log^2|\l|} )_{\l \in \L} \mapsto f(z)e^{-(1 + \e)\a \log^2|z|}, \quad f \in \F^{\infty, 0}_{(1 + \e)\a}
            \end{equation*}
            are uniformly bounded functionals from a subspace of $c_0(\L)$ to $\C$ with $\|T_z\| \le K$. Since $\L$ is SS for $\F^{\infty}_{(1 + \e)\a}$, this subspace is closed, so we can extend $T_z$ to the entire $c_0(\L)$ without increasing the norm. Hence, $T_z \in c_0(\L)^* = \ell^1(\L)$, and so there exists $g_z \in \ell^1(\L)$ such that $\|g_z\|_1 \le K$ and
            \begin{equation*}
                f(z)e^{-(1 + \e)\a \log^2|z|} = \sum_{\l \in \L} f(\l)e^{-(1 + \e)\a \log^2|\l|} g_z(\l), \quad f \in \F^{\infty, 0}_{(1 + \e)\a}.
            \end{equation*}
            Consider $\G = ( \g_n )_{n \in \Z} =  (e^{\frac{n + 1 - 2/p}{2 \e \a}} )_{n \in \Z}$ and a function
            \begin{equation*}
                G(z) = \prod_{k = 0}^{\infty} \left(1 - \frac{\g_{-k}}{z} \right) \prod_{n = 1}^{\infty} \left(1 - \frac{z}{\g_n} \right).
            \end{equation*}
            Estimating both products, see \cite[Lemma 2]{SmallFockCIS}, we get
            \begin{equation*}
                |G(z)| \asymp e^{\e \a \log^2|z|} \frac{d_{\log}(z, \G)}{|z|^{3/2 - 2/p}}, \qquad |G'(\g)| \asymp e^{\e \a \log^2|\g|} \frac{1}{|\g|^{5/2 - 2/p}}, \quad \g \in \G,
            \end{equation*}
            For a $z \in \C_0$, we denote by $\g_z$ an element of $\G$ such that $d_{\log}(z, \G) = d_{\log}(z, \g_z)$. So, $|z/\g_z| \asymp 1$ and $d_{\log}(z, \g_z) \asymp \left| 1 - z/\g_z \right|$. We consider a function
            \begin{equation*}
                P_z(w) = \frac{G(w)}{(w - \g_z)G'(\g_z)} \frac{w^2}{z^2}.
            \end{equation*}
            We have
            \begin{equation*}
                |P_z(w)| \asymp e^{\e \a (\log^2|w| - \log^2|\g_z|)} \frac{|w|^{1/2 + 2/p}}{|z|^{-1/2 + 2/p}} \frac{d_{\log}(w, \G)}{|w - \g_z|}.
            \end{equation*}
            Now, given $f \in \F^p_\a$, from Corollary \ref{crl: eval bound} we have 
            \begin{equation*}
                |f(w)| \lesssim \frac{e^{\a \log^2|w|}}{|w|^{2/p}}.
            \end{equation*}
            Hence, $f P_z \in \F^{\infty, 0}_{(1 + \e)\a}$, and so we get
            \begin{equation*}
                f(w) P_z(w) e^{-(1 + \e)\a \log^2|w|} = \sum_{\l \in \L} f(\l) P_z(\l) e^{-(1 + \e)\a \log^2|\l|} g_w(\l).
            \end{equation*}
            Taking $w = z$, and noting
            \begin{equation*}
                |P_z(z)| \asymp e^{\e \a (\log^2|z| - \log^2|\g_z|)},
            \end{equation*}
            we get
            \begin{equation*}
                |f(z)| e^{-\a \log^2|z|} \lesssim \sum_{\l \in \L} |f(\l)| e^{-\a \log^2|\l|} \frac{|\l|^{1/2 + 2/p}}{|z|^{-1/2 + 2/p}} \frac{d_{\log}(\l, \G)}{|\l - \g_z|} |g_z(\l)|.
            \end{equation*}
            For $1 < p < \infty$, by the Hölder's inequality with $p, q > 1: \ 1/p + 1/q = 1$,
            \begin{align*}
                |f(z)|^p e^{-p\a \log^2|z|} & \lesssim \left( \sum_{\l \in \L} |f(\l)|^p e^{-p\a \log^2|\l|} \frac{|\l|^{p/2 + 2}}{|z|^{-p/2 + 2}} \frac{d_{\log}(\l, \G)^p}{|\l - \g_z|^p} \right) \left( \sum_{\l \in \L} |g_z(\l)|^q \right)^{p/q} \\
                & \le \left( \sum_{\l \in \L} |f(\l)|^p e^{-p\a \log^2|\l|} \frac{|\l|^{p/2 + 2}}{|z|^{-p/2 + 2}} \frac{d_{\log}(\l, \G)^p}{|\l - \g_z|^p} \right) \left( \sum_{\l \in \L} |g_z(\l)| \right)^p \\
                & \le K^p \sum_{\l \in \L} |f(\l)|^p |\l|^2 e^{-p\a \log^2|\l|} \frac{|\l|^{p/2}}{|z|^{-p/2 + 2}} \frac{d_{\log}(\l, \G)^p}{|\l - \g_z|^p}.
            \end{align*}
            As for $0 < p \le 1$, by the sub-additivity of $x^p$, we have
            \begin{align*}
                |f(z)|^p e^{-p\a \log^2|z|} & \lesssim \sum_{\l \in \L} |f(\l)|^p e^{-p\a \log^2|\l|} \frac{|\l|^{p/2 + 2}}{|z|^{-p/2 + 2}} \frac{d_{\log}(\l, \G)^p}{|\l - \g_z|^p} |g_z(\l)|^p \\
                & \le K^p \sum_{\l \in \L} |f(\l)|^p e^{-p\a \log^2|\l|} \frac{|\l|^{p/2 + 2}}{|z|^{-p/2 + 2}} \frac{d_{\log}(\l, \G)^p}{|\l - \g_z|^p} 
            \end{align*}
            Thus, for any $0 < p < \infty$, integrating over $\C_0$, we get
            \begin{align*}
                \| f \|^p_{\F^p_{\a}} & \lesssim K^p \left( \sum_{\l \in \L} |f(\l)|^p |\l|^2 e^{-p\a \log^2|\l|} \int_{\C_0} \frac{|\l|^{p/2}}{|z|^{-p/2 + 2}} \frac{d_{\log}(\l, \G)^p}{|\l - \g_z|^p} dA(z) \right) \\
                & \le K^p \|f|_\L\|^p_{p, \a} \sup_{\l \in \L} \int_{\C_0} \frac{|\l|^{p/2}}{|z|^{-p/2 + 2}} \frac{d_{\log}(\l, \G)^p}{|\l - \g_z|^p} dA(z).
            \end{align*}
            It follows that it is enough to show
            \begin{equation*}
                \int_{\C_0} \frac{|w|^{p/2}}{|z|^{-p/2 + 2}} \frac{d_{\log}(w, \G)^p}{|w - \g_z|^p} dA(z) \lesssim 1, \quad w \in \C_0.    
            \end{equation*}
            We estimate two parts separately:
            \begin{align*}
                \int_{|\g_z| \ge 2 |w|} \frac{|w|^{p/2}}{|z|^{-p/2 + 2}} \frac{d_{\log}(w, \G)^p}{|w - \g_z|^p} dA(z) & \lesssim \int_{|\g_z| \ge 2 |w|} |w|^{p/2} \frac{1}{|z|^{p/2 + 2}} dA(z) \\ 
                & \le \int_{|z| \gtrsim |w|} |w|^{p/2} \frac{1}{|z|^{p/2 + 2}} dA(z) \lesssim 1,
            \end{align*}
            \begin{align*}
                \int_{|\g_z| < 2 |w|} \frac{|w|^{p/2}}{|z|^{-p/2 + 2}} \frac{d_{\log}(w, \G)^p}{|w - \g_z|^p} dA(z) & \lesssim \int_{|\g_z| < 2 |w|} \frac{|w|^{-p/2}}{|z|^{-p/2 + 2}} \frac{\dist(w, \G)^p}{|w - \g_z|^p} dA(z) \\
                & \le \int_{|\g_z| < 2 |w|} \frac{|w|^{-p/2}}{|z|^{-p/2 + 2}} dA(z) \\
                & \le \int_{|z| \lesssim  |w|} \frac{|w|^{-p/2}}{|z|^{-p/2 + 2}} dA(z) \lesssim 1.
            \end{align*}
            This completes the proof.
        \end{proof}
        
\section{Proof of Theorem \ref{thrm: ShS p = 2}} \label{sect: Theorem 2}
    \subsection{Necessity} 
        Suppose that a separated $\L$ is ShS for $\F^p_\a$. We are going to show that it is ShS for $\F^{\infty}_\a$ and therefore, by Theorem \ref{thrm: ShS p = infty}, $D^-_{\log}(\L) > 2\a$.
        
        By Lemma \ref{lem: weak uniqueness} it is enough to show that any $\tilde \L \in W(\L)$ is a uniqueness set for $\F^{\infty}_\a$. Suppose it is not. Then there exists a non-zero $f \in \F^{\infty}_\a$ such that $f(\tilde \l) = 0, \ \tilde \l \in \tilde \L$. Take $\tilde \l_0, \ldots, \tilde \l_n \in \tilde \L$ with $n \ge 2/p$ and consider
        \begin{equation*}
            g(z) = \frac{f(z)}{(z - \tilde \l_0)\ldots(z - \tilde \l_n)}.
        \end{equation*}
        It belongs to $\F^p_{\a}$ and vanishes on $\tilde \L \setminus \{ \tilde \l_0, \ldots, \tilde \l_n \}$. But by Lemma \ref{lem: weak sampling p = 2} $\tilde \L$ is ShS for $\F^p_\a$, and therefore, by Lemma \ref{lem: removing point sampling}, $\tilde \L \setminus \{ \tilde \l_0, \ldots, \tilde \l_n \}$ is SS for $\F^p_\a$, a contradiction to the fact that non-zero $g \in \F^p_{\a}$ vanishes on $\tilde \L \setminus \{ \tilde \l_0, \ldots, \tilde \l_n \}$. 
    \subsection{Sufficiency} 
        Suppose $D^-_{\log}(\L) > 2\a$. Then, for sufficiently small $\e > 0$, we have $D^-_{\log}(\L) > 2(1 + \e)\a$. Hence, by Theorem \ref{thrm: ShS p = infty}, $\L$ is ShS for $\F^{\infty}_{(1 + \e)\a}$. Thus, by Lemma \ref{lem: alpha+eps, p = infty to alpha, p = 2}, $\L$ is ShS for $\F^p_\a$, completing the proof. \hfill \qedsymbol

    \subsection*{Acknowledgment} 
        The authors are grateful to Anton Baranov, Evgeny Abakumov, and Alexander Borichev for their valuable comments and suggestions.

\printbibliography

\end{document}